\pdfoutput=1
\documentclass{amsart}

\usepackage{amssymb,amsfonts}
\usepackage{enumerate}
\usepackage{mathrsfs}
\usepackage{fullpage}
\usepackage{tikz}
\usepackage{tikz-cd}
\usepackage{spectralsequences}
\usepackage{stmaryrd}
\usepackage{hyperref}
\usepackage{todonotes}

\newtheorem{thm}{Theorem}[section]
\newtheorem{cor}[thm]{Corollary}
\newtheorem{prop}[thm]{Proposition}
\newtheorem{lem}[thm]{Lemma}
\newtheorem{conj}[thm]{Conjecture}
\newtheorem{quest}[thm]{Question}
\newtheorem*{claim}{Claim}

\theoremstyle{definition}
\newtheorem{defn}[thm]{Definition}

\newtheorem{exmp}[thm]{Example}

\theoremstyle{remark}
\newtheorem{rem}[thm]{Remark}
\newtheorem{rems}[thm]{Remarks}

\newcommand{\CAlg}{\mathsf{CAlg}}
\newcommand{\Mod}{\mathsf{Mod}}
\DeclareMathOperator{\Hom}{Hom}
\DeclareMathOperator{\Der}{Der}
\DeclareMathOperator{\Ext}{Ext}
\DeclareMathOperator{\SP}{SP}
\DeclareMathOperator{\Map}{Map}
\DeclareMathOperator{\Sq}{Sq}
\DeclareMathOperator{\Tor}{Tor}
\newcommand{\abs}[1]{\left\lvert #1 \right\rvert}
\DeclareMathOperator{\Ann}{Ann}
\DeclareMathOperator{\im}{im}
\DeclareMathOperator{\colim}{colim}
\newcommand{\tr}{\mathrm{tr}}
\newcommand{\gen}[1]{\left\langle #1 \right\rangle}
\DeclareMathOperator{\Stab}{Stab}
\DeclareMathOperator{\Ind}{Ind}
\DeclareMathOperator{\Cl}{Cl}
\DeclareMathOperator{\Gr}{Gr}
\DeclareMathOperator{\fib}{fib}
\DeclareMathOperator{\THH}{THH}

\makeatletter
\let\c@equation\c@thm
\makeatother
\numberwithin{equation}{section}

\title{Strict units of commutative ring spectra}

\author{Jun Hou Fung}
\address{Department of Mathematics, Harvard University, Cambridge, MA 02138}
\email{jhfung@math.harvard.edu}

\date{\today}

\begin{document}

\begin{abstract}
	We provide computational tools to calculate the strict units of commutative ring spectra.  We describe the Goerss-Hopkins-Miller spectral sequence for computing strict units of $E_\infty$-$H\mathbb{F}_p$-algebras, and use it to compute the strict units of polynomial and truncated polynomial rings, whose Postnikov towers we also analyze.  We then sketch the calculation of the strict units of other rings, such as $S^0$, using the symmetric product of spheres filtration and transfers.  
\end{abstract}

\maketitle

\tableofcontents

\section{Introduction}
	In classical algebra, we may isolate the multiplicative operation in a ring $R$ by passing to its group of units $R^\times$ consisting of those elements in $R$ which are invertible.  In stable homotopy theory, we have analogues of rings -- ring spectra -- which have a beautiful but much more complicated theory of multiplication.   This paper explores this theory for structured ring spectra from the perspective of its units.  
	
	Let $R$ be an $E_\infty$-ring spectrum.  Its zeroth space $\Omega^\infty R$ is an $E_\infty$-ring space, and we define:
\begin{defn}[\cite{MQRT, ABGHR}]
	The space $GL_1 R$ of units of $R$ is the pullback in the square
	\begin{center}
		\begin{tikzcd}
			GL_1 R \ar[r] \ar[d] \ar[rd, phantom, very near start, "\lrcorner"] & \Omega^\infty R \ar[d, "\pi_0"] \\
			(\pi_0 R)^\times \ar[r] & \pi_0 R
		\end{tikzcd}
	\end{center}
	
	That is, $GL_1 R$ is the union of those path components of $\Omega^\infty R$ corresponding to invertible elements in $\pi_0 R$. 
\end{defn}

	The space $GL_1 R$ has the structure of a grouplike $E_\infty$-space coming from the multiplication in $\Omega^\infty R$, and consequently gives rise to a connective spectrum denoted $gl_1 R$.  
	
	If we think of $R$ as describing a multiplicative cohomology theory, then $gl_1 R$ represents its units in the sense that if $X$ is a space, then we have \[[X, gl_1 R] \cong (R^0 X)^\times.\]  Another reason $gl_1 R$ deserves to be called the units of $R$ is that there is an equivalence 
	\begin{equation}
		\label{eqn:gl1Defn}
		gl_1 R \simeq \CAlg(S^0 \{t^\pm\}, R),
	\end{equation} 
	\noindent where $S^0 \{t^\pm\} := \Sigma^\infty_+ QS^0$ is the free $E_\infty$-ring with an invertible generator $t$ in degree $0$.  
	
	These units of $E_\infty$-ring spectra play prominent roles in various areas of homotopy theory, such as orientation theory and the theory of Thom spectra \cite{MQRT, ABGHR, ABG18}.

	There is another variant of these units introduced by Hopkins and Lurie in the setting of spectral algebraic geometry \cite{LurSAG} by replacing $S^0\{t^\pm\}$ with $S^0[t^\pm] := \Sigma^\infty_+ \mathbb{Z}$ in \eqref{eqn:gl1Defn}.  
	
\begin{defn}
	The \emph{strictly commutative units} of $R$, or \emph{strict units} for short, is the mapping space \[\mathbb{G}_m(R) := \CAlg(S^0\{t^\pm\}, R).\]
\end{defn}

	By adjunction, we also have \[\mathbb{G}_m(R) \simeq \Map(H\mathbb{Z}, gl_1 R).\]  
	
	One reason we call these elements strict units is because they arise as the units of a strictly commutative (i.e., simplicial commutative) ring.  Another reason is because we have replaced $QS^0$ with the discrete $E_\infty$-monoid $\mathbb{Z} \cong \pi_0 QS^0$, which doesn't support much higher homotopical structure.  
	
	Intuitively, one can imagine these strictly commutative elements coming from the rigidity that is imposed by ``geometry''.  One starting point of our exploration of strict units is the calculation of the strict units of complex $K$-theory.  Note that while $\pi_* gl_1 R$ only depends on the homotopy groups of $R$, the same is not true for $\mathbb{G}_m(R)$.  In any case, the units of $K$-theory ought to have to do with tensor-invertible vector bundles on a space $X$.  Geometrically, there is a natural source for such bundles, namely the group of virtual line bundles on $X$.  How could we locate this particular collection of bundles within $gl_1 K$, whose homotopy groups are supported in infinitely many degrees?  More precisely, we have \[GL_1 K \simeq \mathbb{Z}/2 \times BU(1) \times \tau_{\geq 4} GL_1 K,\] and we want a natural way to pick out the first two factors.  It turns out that the strict units pick out this collection of line bundles. Let $p$ be a prime and write $K_{p^\infty}$ for $p$-adic $K$-theory.  Then, \[\pi_* \mathbb{G}_m(K_{p^\infty}) \cong \begin{cases} \overline{\mathbb{F}_p}^\times, & * = 0 \\ \mathbb{Z}_p, & * = 2 \\ 0, & \text{otherwise}. \end{cases}\]
	
	More generally, we have
\begin{thm}[Hopkins-Lurie, Sati-Westerland \cite{SW15}]
	Let $E$ be a Lubin-Tate theory of height $h$ with residue field $\overline{\mathbb{F}_p}$.  Then, \[\pi_* \mathbb{G}_m(E) \cong \begin{cases} \overline{\mathbb{F}_p}^\times, & * = 0 \\ \mathbb{Z}_p, & * = h + 1 \\ 0, & \text{otherwise}. \end{cases}\]
\end{thm}

	Observe this answer is much simpler than the homotopy groups of $gl_1 E$.  
	
	Our goal in this paper is to provide calculations of strict units for other $E_\infty$-rings, of a different flavor.  Before describing our results, let us first give some additional motivation for studying strict units, in addition to picking out geometrically interesting units.
	
\begin{enumerate}[(1)]
	\item Strict units parametrize $E_\infty$-twists of multiplicative cohomology theories by ordinary cohomology.  
	
	As a concrete example, let us recall the construction of twisted $K$-theory, following Atiyah-Segal \cite{AS04}.  Atiyah \cite{Ati69} showed that $K$-theory is represented by the space $\mathcal{F}$ of Fredholm operators on a complex infinite-dimensional Hilbert space $\mathcal{H}$, i.e., by sections of the trivial $\mathcal{F}$-bundle over a space $X$.  We get twisted $K$-theory by looking at sections of twisted bundles.  More precisely, if $P$ is a principal $PU(\mathcal{H})$-bundle, then we can construct the $\mathcal{F}$-bundle $P \times_{PU(\mathcal{H})} \mathcal{F} \to X$, and define the $P$-twisted $K$-theory of $X$ by \[K(X)_P = \Gamma(P \times_{PU(\mathcal{H})} \mathcal{F} \to X).\]  The twists of $K$-theory arising in this way are classified by $[X, BPU(\mathcal{H})]$, i.e., by a class in $H^3(X;\mathbb{Z})$ since $BPU(\mathcal{H})$ is a $K(\mathbb{Z},3)$.  
	
	Later on, May-Sigurdsson \cite{MS06} and Ando-Blumberg-Gepner \cite{ABG10} generalized this construction to other cohomology theories.  From this perspective, the main ingredient in twisted $K$-theory is a map $K(\mathbb{Z},3) \to BGL_1 K$ induced by the Atiyah-Bott-Shapiro orientation of $\operatorname{Spin}^c$.  
	
	In general, twistings of cohomology theories by ordinary cohomology are parametrized by maps $K(\mathbb{Z},n) \to GL_1 R$, and $E_\infty$-twistings correspond to strict units of $R$.  See \cite{SW15} for examples of twisted Morava $K$-theory and $E$-theory.  
	
	
	\item Strict units gives rise to extensions of commutative ring spectra by formally adjoining roots \cite{Law19}.  Given a map $\alpha: H\mathbb{Z} \to gl_1(R)$, i.e., an $E_\infty$-ring map $S^0[t^\pm] \to R$, we can construct \[R[\sqrt[n]{\alpha}] := S^0[t^{\pm 1/n}] \wedge_{S^0[t^\pm]} R\] which, as its name suggests, is the ring $R$ with a formal $n$-th root of $\alpha$ adjoined.  	
\end{enumerate}
	
\subsection*{Outline}
	We begin in the section \ref{sec:GHM} by describing the Goerss-Hopkins-Miller spectral sequence for computing the strict units of an $E_\infty$-$H\mathbb{F}_p$-algebra.  The Goerss-Hopkins spectral sequence works extremely well when computing for example the strict units of Morava $E$-theory since it degenerates, but is rather more complicated even for $H\mathbb{F}_p$-algebras.  Nonetheless, we work through the relevant homological algebra along the lines of Priddy and Miller and give a description of the $E_2$-page of this spectral sequence in theorems \ref{thm:UnstableCokoszul} and \ref{thm:OddPCokoszul}.  We then quickly specialize to the case $p = 2$ and the ring $H\mathbb{F}_2[u]$ with $\deg u = 2$, and compute parts of the spectral sequence as an example.  
	
	In section \ref{sec:Postnikov}, we continue analyzing the strict units of $H\mathbb{F}_2[u]$ through its Postnikov tower.  Using a splitting of $gl_1 H\mathbb{F}_2[u]$ due to Steiner, and an observation of Lawson and Mathew-Stojanoska on the bottom $k$-invariant of $gl_1 H\mathbb{F}_2[u]$, we can describe the Postnikov tower of $gl_1 H\mathbb{F}_2[u]$ (\ref{prop:g1Postnikov}) and the associated (modified) Atiyah-Hirzebruch spectral sequence (\ref{prop:g1AHSS}).  We conclude this section with a description of the homotopy groups of the strict units of truncated polynomial rings $H\mathbb{F}_2[u]/u^{n+1}$ in theorem \ref{thm:table} and an elementary conjecture (\ref{conj:AdemHomology}) about the relations in Steenrod algebra that would give a computation of the entire structure of the strict units of $H\mathbb{F}_2[u]$.  We also generalize slightly the result on $k$-invariants in the odd primary case in (theorem \ref{thm:OddKInvt}).  
	
	In section \ref{sec:SymProd}, we outline a method to compute strict units in the cases when the homological algebra necessary for the Goerss-Hopkins-Miller spectral sequence is not available.  This includes the case of computing the strict units of the sphere spectrum, for example.  The key idea is to use the symmetric product of spheres filtration to provide a resolution of $H\mathbb{Z}$ by ``spacelike'' spectra.  This is the resolution that is featured in the Whitehead conjecture, and so the $E_2$ page of our spectral sequence is essentially governed by the solution of the Whitehead conjecture in other generalized cohomology theories.  We demonstrate the utility of this method by re-computing the $K$-theory and $j$-theory of $H\mathbb{Z}$ (theorem \ref{thm:K0MsLs} and corollary \ref{cor:HZj}) by reducing it to a problem in representation theory -- in these cases, verification of the appropriate form of Whitehead conjecture is sufficient to establish the collapse of the spectral sequence.  Similarly, we analyze the symmetric product of spheres spectral sequence in the case of $gl_1 S^0$ by recasting it in terms of the Burnside rings in theorem \ref{thm:LsGL1S0}.  We also indicate some possible approaches to understanding the rest of the spectral sequence by comparing the additive and multiplicative versions of the spectral sequence.  

	Finally, we provide an appendix to prove a result (theorem \ref{thm:DLStrs}) about the possible Dyer-Lashof algebra structures on the polynomial ring $\mathbb{F}_2[u]$, motivated by a question of Totaro's.  
	
\subsection*{Acknowledgements}
	I thank my advisor Mike Hopkins for suggesting this topic, for many illuminating conversations, and for his encouragement throughout this project.  I also thank Jeremy Hahn and Haynes Miller for useful discussions, and Daniel Hess for directing me to his thesis \cite{Hes19}.  
	
\section{The Goerss-Hopkins-Miller spectral sequence for \texorpdfstring{$E_\infty$}{E-infinity}-mapping spaces}
\label{sec:GHM}
	
\subsection{The homotopy spectral sequence}
	The strict units of an $E_\infty$-ring $R$ is the mapping space \[\mathbb{G}_m(R) = \CAlg(\Sigma^\infty_+ \mathbb{Z}, R).\]  As such, we can use the structured version of the homotopy spectral sequence discussed in \cite{GH04} to compute it.  For tractibility, we write $H = H\mathbb{F}_p$ and assume that $A$ is an $E_\infty$-$H$-algebra.  In this case, we have \[\CAlg(\Sigma^\infty_+ \mathbb{Z}, A) \simeq \CAlg_H(H \wedge \Sigma^\infty_+ \mathbb{Z}, A).\]   
	
	We begin by briefly sketching the form of the spectral sequence we use.  The forgetful functor from $E_\infty$-$H$-algebras to $H$-modules has a left adjoint $\mathbb{P} = \mathbb{P}_H: \Mod_H \to \CAlg_H$ given by \[\mathbb{P}(M) = \bigoplus_{n \geq 0} E\Sigma_{n+} \wedge_{\Sigma_n} M^{\otimes n}.\]  We shall write $\mathbb{P}: \CAlg_H \to \CAlg_H$ also for the comonad associated to this adjunction.  
	
	We resolve $H \wedge \Sigma^\infty_+ \mathbb{Z}$ using $\mathbb{P}$, and apply $\CAlg_H(-, A)$ to the resolution.  This begets a cosimplicial space $\CAlg_H(\mathbb{P}^{\bullet + 1}(H \wedge \Sigma^\infty_+ \mathbb{Z}), A)$.  In good cases, its totalization can be identified with $\CAlg_H(H \wedge \Sigma^\infty_+ \mathbb{Z}, A)$, and we have a Bousfield-Kan spectral sequence \[E_1^{s,t} = \pi^s \pi_t \CAlg_H(\mathbb{P}^{\bullet + 1}(H \wedge \Sigma^\infty_+ \mathbb{Z}), A) \Rightarrow \pi_{t-s} \CAlg_H(H \wedge \Sigma^\infty_+ \mathbb{Z}, A).\]  
	
	The $E_2$-page has an admittedly not very useful name in terms of cotriple homology: \[E_2^{s,t} \cong \mathscr{R}_\mathbb{P}^s \pi_t \CAlg_H(-,A)(H \wedge \Sigma^\infty_+ \mathbb{Z}).\]  However, as usual there is an identification of this $E_2$-term purely in terms of algebra.  
	
	We study $\CAlg_H(\mathbb{P}M, A)$ for an $H$-module $M$.  Let $\mathcal{R}$ be the (mod $p$) Dyer-Lashof algebra.  First, if $A$ is an $E_\infty$-$H$-algebra, then $\pi_* A$ is a commutative algebra with an allowable action by $\mathcal{R}$.  Furthermore, we have the classical computation:
\begin{prop}[{\cite[I.4]{CLM76}}]
	$\pi_* \mathbb{P}M$ is the free allowable algebra over $\mathcal{R}$ generated by $\pi_* M$.  
\end{prop}

	Write $F$ for the free allowable algebra over $\mathcal{R}$ functor on graded $\mathbb{F}_p$-vector spaces, so that we have $\pi_* \mathbb{P}M \cong F(\pi_* M)$.  Using this we deduce that
	
\begin{prop}
	The map $\pi_0 \CAlg_H(\mathbb{P}M, A) \to \Hom_{\CAlg_{\mathbb{F}_p, \mathcal{R}}}(F(\pi_* M), \pi_* A)$ is an isomorphism.  
\end{prop}

	Assume we are given a base point of $\CAlg_H(\mathbb{P}M, A)$, i.e., a $E_\infty$-map $\phi: \mathbb{P}M \to A$.  Recall that if $A$ is an $E_\infty$-$H$-algebra, then $A^{S^t}$ also acquires the structure of an $E_\infty$-$H$-algebra.  Suppose $t > 0$.  Then as a corollary of the previous proposition and the fact that $\pi_* A^{S^t} \cong \pi_* A[\epsilon_t]/\epsilon_t^2$ is a degree $t$ square-zero extension of $\pi_* A$, we have \[\pi_t(\CAlg_H(\mathbb{P}M, A), \phi) \cong \Der_\mathcal{R}(F(\pi_* M), \Omega^t \pi_* A),\] where $\Der_\mathcal{R}$ is the module of commutative $\mathbb{F}_p$-algebra derivations that commute with the Dyer-Lashof operations, and $\Omega^t \pi_* A$ is an $\pi_* \mathbb{P}M$-module via $\pi_* \phi$.  From now on we suppress the choice of basepoint.  
	
	Note that $\pi_*(H \wedge \Sigma^\infty_+ \mathbb{Z}) \cong \mathbb{F}_p[u^\pm]$.  We have \[\pi_t \CAlg(\mathbb{P}^{\bullet + 1}(H \wedge \Sigma^\infty_+ \mathbb{Z}), A) \cong \Der_\mathcal{R}(F^{\bullet + 1}(\mathbb{F}_p[u^\pm]), \Omega^t \pi_* A).\]  Thus the $E_2$-term of the homotopy spectral sequence is given by the right-derived functors of $\Der_\mathcal{R}$: \[E_2^{s,t} \cong \mathscr{R}_F^s(\Der_\mathcal{R}(-, \Omega^t \pi_* A))(\mathbb{F}_p[u^\pm]).\]  We set \[D_\mathcal{R}^s(\Gamma, A) := \mathscr{R}_F^s(\Der_\mathcal{R}(-,A))(\Gamma).\]  Of course, if $t = 0$, then suitable adjustments need to be made, e.g., \[E_2^{0,0} \cong \Hom_{\CAlg_{\mathbb{F}_p, \mathcal{R}}}(\mathbb{F}_p[u^\pm], \pi_* A).\]  
	
\begin{rem}
	Compared with \cite{GH04}, our task here is substantially simplified since we don't have to resolve any operads; we have a complete understanding of $\pi_* \mathbb{P}M$ solely in terms of $\pi_* M$.  
\end{rem}

	Following a well-trodden path, the next step is to give a composite functor spectral sequence computing these right-derived functors.  
	
\begin{lem}[{\cite[Lem.~6.3]{GH04}}]
	Let $\Gamma$ be an allowable algebra over the Dyer-Lashof algebra.  Then the graded module $\Omega_{\Gamma/\mathbb{F}_p}$ of commutative algebra derivations is an object in the category $a_1 \Mod_{\mathcal{R}, \Gamma}$ of unstable $\Gamma$-modules over the Dyer-Lashof algebra, and there is a natural isomorphism \[\Der_\mathcal{R}(\Gamma, M) \cong \Hom_{a_1 \Mod_{\mathcal{R}, \Gamma}}(\Omega_{\Gamma/\mathbb{F}_p}, M).\]  Furthermore, if we are given an augmentation $\Gamma \to \mathbb{F}_p$ and $M \in a_1 \Mod_{\mathcal{R}, \Gamma}$ is obtained from $\mathcal{U} := a_1 \Mod_{\mathcal{R}, \mathbb{F}_p}$ by restriction of scalars, then \[\Der_\mathcal{R}(\Gamma, M) \cong \Hom_\mathcal{U}(\mathbb{F}_p \otimes_\Gamma \Omega_{\Gamma/\mathbb{F}_p}, M).\]  
\end{lem}

\begin{rem}
	If $\pi_* A$ is connected in the sense $\pi_0 A \cong \mathbb{F}_p$, then $\pi_0 \phi$ furnishes an augmentation of $\mathbb{F}_p[u^\pm]$ such that the $\mathbb{F}_p[u^\pm]$-module structure on $\pi_* A$ is obtained by restriction of scalars.  The same is true also for $\Omega^t \pi_* A$.  
\end{rem}

	In this case we have a composition of functors \[a_0 \CAlg_\mathcal{R} \xrightarrow{\mathbb{F}_p \otimes_{(-)} \Omega_{(-)/\mathbb{F}_p}} \mathcal{U} \xrightarrow{\Hom_\mathcal{U}(-, M)} \Mod_{\mathbb{F}_p},\] which gives rise to the following Grothendieck spectral sequence

\begin{prop}[{\cite[Prop.~6.4]{GH04}}]
	There is a composite functor spectral sequence \[\Ext_\mathcal{U}^p(D_q(\Gamma / \mathbb{F}_p, \mathbb{F}_p), M) \Rightarrow D_\mathcal{R}^{p+q}(\Gamma, M).\]  
\end{prop}
	
	\noindent Here, $D_*(\Gamma /\mathbb{F}_p, \mathbb{F}_p)$ is the ordinary Andr\'{e}-Quillen homology of $\Gamma$ with coefficients in $\mathbb{F}_p$.  
	
	In our case of interest, we want to compute $D_*(\mathbb{F}_p[u^\pm] / \mathbb{F}_p, \mathbb{F}_p)$.  
	
\begin{prop}
	$D_q(\mathbb{F}_p[u^\pm]/\mathbb{F}_p, \mathbb{F}_p) \cong \mathbb{F}_p$ is $q = 0$, and is $0$ otherwise.  
\end{prop}
\begin{proof}
	Since $\mathbb{F}_p[u^\pm]$ is smooth over $\mathbb{F}_p$, the cotangent complex $\mathbb{L}_{\mathbb{F}_p[u^\pm]/\mathbb{F}_p} \simeq \Omega_{\mathbb{F}_p[u^\pm] /\mathbb{F}_p}$ is concentrated in degree zero.  We have a transitivity sequence \[\mathbb{L}_{\mathbb{F}_p[u]/\mathbb{F}_p} \otimes_{\mathbb{F}_p[u]} \mathbb{F}_p[u^\pm] \to \mathbb{L}_{\mathbb{F}_p[u^\pm]/\mathbb{F}_p} \to \mathbb{L}_{\mathbb{F}_p[u^\pm]/\mathbb{F}_p[u]}.\]  The last term vanishes because $\mathbb{F}_p[u^\pm]$ is \'{e}tale over $\mathbb{F}_p[u]$.  Furthermore, $\mathbb{L}_{\mathbb{F}_p[u]/\mathbb{F}_p}$ is free of rank one.  So $\mathbb{L}_{\mathbb{F}_p[u^\pm]/\mathbb{F}_p} \simeq \mathbb{F}_p[u^\pm]$.  
	
	Hence, \[D_q(\mathbb{F}_p[u^\pm]/\mathbb{F}_p, \mathbb{F}_p) \cong H_q(\mathbb{L}_{\mathbb{F}_p[u^\pm]/\mathbb{F}_p} \otimes_{\mathbb{F}_p[u^\pm]} \mathbb{F}_p) \cong \begin{cases} \mathbb{F}_p, & q = 0 \\ 0, & \text{otherwise} \end{cases}\] as wanted.  
\end{proof}

	As a consequence, the composite functor spectral sequence degenerates, and the $E_2$-page of the homotopy spectral sequence is \[D_\mathcal{R}^s(\mathbb{F}_p[u^\pm], \Omega^t \pi_* A) \cong \Ext_\mathcal{U}^s(\mathbb{F}_p, \Omega^t \pi_* A).\]  
	
\subsection{Koszul resolutions}
	Following Priddy \cite{Pri70} and Miller \cite{Mil78}, we construct a complex computing $\Ext_\mathcal{U}$.  We shall first sketch the general theory of co-Koszul complexes for computing $\Ext$ over a homogeneous Koszul algebra, then recall the cohomology of the Dyer-Lashof algebra, and finally incorporate the unstability conditions to obtain a complex computing $\Ext_\mathcal{U}$.  Our main result in this section is theorem \ref{thm:UnstableCokoszul}, but we will first have to introduce some notation before stating it.  
	
\subsubsection{Ext over homogeneous Koszul algebras}
	Let $A$ be a homogeneous Koszul algebra over a field $k$, and $M$ a (left) $A$-module.  Let $\{a_i\}_{i \in I}$ be a set of Koszul generators for $A$, and let the admissible relations for $A$ be given by \[a_i a_j = \sum_{(k,l) \in S} f(i,j;k,l) a_k a_l.\]  
	
\begin{defn}[{\cite[Thm.~4.6]{Pri70}}]
\mbox{}
\begin{enumerate}[(a)]
	\item The \emph{co-Koszul complex} $\bar{K}^*(A)$ computing $\Ext_A(k,k)$ is the differential algebra 
	\mbox{}
	\begin{itemize}
		\item generated by $\{\beta_i\}_{i \in I}$ with $\deg \beta_i = (1,i)$,
		\item for each $(i,j) \in S$, a relation \[(-1)^{\nu_{i,j}} \beta_i \beta_j + \sum_{k,l \notin S} (-1)^{\nu_{k,l}} f(k,l;i,j) \beta_k \beta_l,\] where $\nu_{u,v} = \deg \beta_u + (\deg \beta_u - 1)(\deg \beta_v - 1)$,
		\item and zero differential (since $A$ is homogeneous).  
	\end{itemize}

	\item The co-Koszul complex computing $\Ext_A(k,M)$ is the differential right $\bar{K}^*(A)$-module $M \widehat{\otimes} \bar{K}^*(A)$, with differential determined by \[d(m \otimes 1) = \sum_{i \in I} (-1)^{\deg \beta_i} a_i m \otimes \beta_i.\]
\end{enumerate}
\end{defn}

\begin{rem}
	Because of the linear dualization and the fact that $A$ and $M$ need not be finite-dimensional, the tensor product in (b) needs to be completed, i.e., we allow infinite linear combinations of the basic tensors.  
	
	This is an artifact of expressing $\Hom$ in terms of a tensor product.  A better description of the co-Koszul complex would be $\Hom_F(\bar{K}_*(A), M)$.
\end{rem}

\begin{rem}
	The differential in (b) is obtained as follows.  Given a left $A$-module $M$, we have $\Hom_A(k,M) \cong (M^\vee \otimes_A F)^\vee$.  Form the cobar construction $(M^\vee \otimes T(I(A)))^\vee$.  According to \cite[Eqn.~4.2]{Pri70}, the differential is obtained from the coaction on $M$ by $A^\vee$.  In terms of $M$, this is obtained by \[M \cong M \otimes k \xrightarrow{1 \otimes \operatorname{cotr}} M \otimes A \otimes A^\vee \cong A \otimes M \otimes A^\vee \to M \otimes A^\vee.\]  We project to the length 1 words in $A^\vee$ to get the formula for the differential.  
\end{rem}

\subsubsection{The cohomology of the Dyer-Lashof algebra}
\label{sec:DL}

	The mod 2 Dyer-Lashof algebra $\mathcal{R}$ is the graded $\mathbb{F}_2$-algebra with generators $\{Q^i\}_{i \geq 0}$, with $\deg Q^i = i$ and Adem relations \[Q^i Q^j = \sum_k \binom{k-j-1}{2k-i} Q^{i+j-k} Q^k, \quad i > 2j.\]  Given a sequence $I = (i_1, \ldots, i_s)$ of non-negative integers, we abbreviate $Q^I := Q^{i_1} \cdots Q^{i_s}$.  Call a sequence $I$ \emph{allowable} if $i_j \leq 2i_{j+1}$.  Then $\{Q^I \mid I \text{ allowable}\}$ forms a basis for $\mathcal{R}$.  Moreover, $\mathcal{R}$ is a pre-Koszul algebra with generators $\{Q^i\}_{i \geq 0}$ and admissible relations as indicated above.  In fact, the basis of allowable monomials is a PBW basis for $\mathcal{R}$; hence $\mathcal{R}$ is a homogeneous Koszul algebra.  
	
	Following the prescription in the previous section, the cohomology \[H^*(\mathcal{R}) := \Ext_\mathcal{R}^*(\mathbb{F}_2, \mathbb{F}_2) \cong \bar{K}^*(\mathcal{R})\] of the Dyer-Lashof algebra can be computed.  
	
\begin{prop}[\cite{KL83}]
	The cohomology of $\mathcal{R}$ is the algebra 
	\begin{itemize}
		\item generated by $\{\beta_i\}_{i \geq 0}$ with $\deg \beta_i = i+1$, 
		\item and for each $i \leq 2j$, a relation \[\beta_i \beta_j = \sum_k \binom{j-k-1}{i-2k-1} \beta_{i+j-k} \beta_k.\]  
	\end{itemize}
\end{prop}

	This cohomology algebra becomes more recognizable if we write the formal symbol $\Sq^{i+1}$ for $\beta_i$.  Then $H^* \mathcal{R}$ is the algebra with generators $\{\Sq^i\}_{i \geq 1}$ and relations \[\Sq^i \Sq^j = \sum_k \binom{j-k-1}{i-2k} \Sq^{i+j-k} \Sq^k\] for $i < 2j$.  In other words, we recover the usual Adem relations for the mod 2 Steenrod algebra, except we have $\Sq^0 = 0$ instead.  
	
	When $p$ is odd, the mod $p$ Dyer-Lashof algebra is the graded $\mathbb{F}_p$-algebra generated by symbols $\{Q^i\}_{i \geq 0} \cup \{\beta Q^i\}_{i \geq 1}$ where $\deg Q^i = 2i(p-1)$ and $\deg \beta Q^i = 2i(p-1) - 1$, subject to the Adem relations \[\beta^\epsilon Q^i Q^j = \sum_k (-1)^{i+k} \binom{(p-1)(k-j) - 1}{pk-i} \beta^\epsilon Q^{i+j-k} Q^k\] for $i > pj$, and \[\beta^\epsilon Q^i \beta Q^j = \sum_k (-1)^{i+k} \binom{(p-1)(k-j)}{pk-i} \beta^\epsilon \beta Q^{i+j-k} Q^k - \sum_k (-1)^{i+k} \binom{(p-1)(k-j)-1}{pk-i-1} \beta^\epsilon Q^{i+j-k} \beta Q^k\] for $i \geq pj$.
	
	Monomials in the generators $\{Q^i\}_{i \geq 0} \cup \{\beta Q^i\}_{i \geq 1}$ also form a PBW basis for $\mathcal{R}$, so $\mathcal{R}$ is again a homogeneous Koszul algebra.  We have the following analogue of the proposition above.
	
\begin{prop}[\cite{KL83}]
	For $p$ an odd prime, $H^*(\mathcal{R})$ is isomorphic to the mod $p$ Steenrod algebra for restricted Lie algebras.  
\end{prop}

	We'll use the notation from the mod $p$ Steenrod algebra to refer to classes in the cohomology of $\mathcal{R}$.  In particular, there is a nonsingular pairing between $\mathcal{R}$ and $\mathcal{A}_L$: \[\langle \beta^\epsilon P^i, \beta^\delta Q^j \rangle = \begin{cases} 1, & \delta + \epsilon = 1 \text{ and } i = j \\ 0, & \text{otherwise}. \end{cases}\]

\subsubsection{Unstability conditions}
	We now specialize to the prime $p = 2$ for expositional convenience; we sketch the (mostly cosmetic) modifications for odd primes in theorem \ref{thm:OddPCokoszul}.  

	Recall that $\mathcal{U}$ is the full subcategory of $\Mod_\mathcal{R}$ consisting of non-negatively graded 1-allowable modules over the Dyer-Lashof algebra.  It is an abelian category on which a projective class is defined by the left adjoint $U: \Mod_{\mathbb{F}_2} \to \mathcal{U}$ to the forgetful functor.  Concretely, $U$ is given by the formula \[U(V) = \mathcal{R} \otimes V / \{Q^i x \mid i \leq \deg x\}.\]  Also, set $R(V)$ to be the free $\mathcal{R}$-module $\mathcal{R} \otimes V$.  
	
	Let $M \in \mathcal{U}$.  Our goal is to compute $\Ext_\mathcal{U}^*(\mathbb{F}_2, M)$, defined as the homology of the standard complex $\Hom_\mathcal{R}(U^{\bullet + 1} \mathbb{F}_2, M)$.  Miller \cite[Prop.~3.1.2]{Mil78} observed that we can use the Koszul technology to compute $\Tor_\mathcal{U}$; in this section we explain how to adapt this for $\Ext_\mathcal{U}$.   
	
	First, $\Hom_\mathcal{R}(U^{\bullet + 1} \mathbb{F}_2, M)$ is a subcomplex of the cobar construction for $R$.  So, $\Ext_\mathcal{U}^*(\mathbb{F}_2, M)$ is the homology of a certain subcomplex of the co-Koszul complex $M \widehat{\otimes} \bar{K}^*(\mathbb{R})$.  
	
	To give an indication of what this subcomplex looks like, recall that the usual bar construction \[M^\vee \widehat{\otimes}_\mathcal{R} R^{\bullet + 1} \mathbb{F}_2\] has a basis consisting of $\{f \otimes Q^{I_0} \otimes Q^{I_1} \otimes \cdots \otimes Q^{I_s}\}$ where each $I_j$ is admissible.  On the other hand, the unstable bar construction $M^\vee \widehat{\otimes}_\mathcal{R} U^{\bullet + 1} \mathbb{F}_2$ has a basis given by $\{f \otimes Q^{I_0} \otimes Q^{I_1} \otimes \cdots Q^{I_s}\}$ where not only is each $I_j$ admissible, but also the following excess conditions are satisfied \[e(I_j) > \abs{I_{j+1}} + \cdots + \abs{I_s}, \quad 0 \leq j < s.\]  Moreover, we have \[M^\vee \otimes_\mathcal{R} U^{\bullet + 1} \mathbb{F}_2 \cong \prod_{d \geq 0} (M^\vee \otimes_\mathcal{R} \mathcal{R}/B(d+1)) \otimes (U^\bullet \mathbb{F}_2)_d,\] where $B(d+1) \subset \mathcal{R}$ is the subspace spanned by those $Q^I$ with $I$ allowable of excess $\leq d$.  This makes sense because
	$B(d+1) \subset \mathcal{R}$ is a two-sided ideal.

	Now we apply the Koszul technology to $U^\bullet \mathbb{F}_2$, or rather to its dual.  Let $L(0) \subset \bar{K}^*(\mathcal{R})$ be the subspace spanned by those $\beta_I := \beta_{i_1} \cdots \beta_{i_s}$ for which $i_j \geq i_{j+1} + \cdots + i_s$ for $1 \leq j \leq s$.  Then $L(0)$ is the image of $(U^\bullet \mathbb{F}_2)^\vee$ under the quotient map $(R^\bullet \mathbb{F}_2)^\vee \to \bar{K}^*(\mathcal{R})$.  Note that $L(0)$ is isomorphic to $\mathcal{A}/\mathcal{A} \Sq^1 \cong H\mathbb{F}_2^*(H\mathbb{Z})$, and we shall use the symbols $\Sq^{i+1}$ in place of the $\beta_i$ from now on.  
	
	We also want to identify $(M^\vee \otimes_\mathcal{R} \mathcal{R}/B(d+1))^\vee$.  We have that \[(M^\vee \otimes_\mathcal{R} \mathcal{R}/B(d+1))^\vee \cong \Hom_\mathcal{R}(\mathcal{R}/B(d+1), M) \cong \Ann_{B(d+1)} M\] is the $\mathcal{R}$-submodule of $M$ consisting of elements that are annihilated by $B(d+1)$.  
	
	To summarize, consider the subspace \[K^*(M^\vee, \mathcal{U}, \mathbb{F}_2) := \prod_{d \geq 0} \Ann_{B(d+1)} M \otimes L(0)_d \subseteq M \widehat{\otimes} \bar{K}^*(\mathcal{R}).\]  
\begin{thm}
	\label{thm:UnstableCokoszul}
	$\Ext_\mathcal{U}^*(\mathbb{F}_2, M)$ can be computed as the homology of the complex $K^*(M^\vee, \mathcal{U}, \mathbb{F}_2)$ with differential \[d(x \otimes 1) = \sum_r Q^r x \otimes \Sq^{r+1}.\] 
\end{thm}

	If $p$ is odd, we let $L(0) \subset \bar{K}^*(\mathcal{R})$ be the subspace spanned by those $\beta^{\epsilon_1} P^{i_1} \beta^{\epsilon_2} P^{i_2} \cdots \beta^{\epsilon_s} P^{i_s}$ for which \[2i_j - \epsilon_j \geq 2(p-1)(i_{j+1} + \cdots + i_s) - (\epsilon_{j+1} + \cdots + \epsilon_s)\] for $1 \leq j \leq s$.  Again, let $B(d+1) \subset \mathcal{R}$ is the subspace spanned by those $\beta^\epsilon Q^I$ with $I$ allowable of excess $<d$.  
	
\begin{thm}
	\label{thm:OddPCokoszul}
	For $p$ odd, $\Ext_\mathcal{U}^*(\mathbb{F}_p, M)$ can be computed as the homology of the complex \[K^*(M^\vee, \mathcal{U}, \mathbb{F}_p) := \prod_{d \geq 0} \Ann_{B(d+1)} M \otimes L(0)_d \subseteq M \widehat{\otimes} \bar{K}^*(\mathcal{R}),\] equipped with the differential \[d(x \otimes 1) = \sum_{r \geq 0} \beta Q^r x \otimes P^r - \sum_{r \geq 0} Q^r x \otimes \beta P^r.\]
\end{thm}


\begin{rem}
	These theorems should be thought of as the ``$\Ext$'' analogues of Miller's results on unstable $\Tor$, and are rather straightforwardly related by linear dualization.  At odd primes, the computation of unstable $\Tor$ is worked out in Hess' thesis \cite{Hes19}.  
\end{rem}

\subsection{Examples}
	In this section, we apply theorem \ref{thm:UnstableCokoszul} to compute some $\Ext_\mathcal{U}$ groups, and hence shed light on strict units of some $E_\infty$-$H$-algebras.  
	
	Our main example is $A = H\mathbb{F}_2[u]$ where $\deg u = 2$; here $H\mathbb{F}_2[u]$ is given the $E_\infty$-structure coming from the multiplication of the ordinary graded ring $\mathbb{F}_2[u]$.  The choice of the degree of $u$ only changes things superficially.
	
\begin{rem}
	The units of the ring spectrum $H\mathbb{F}_2[u]$ was first constructed by Segal \cite{Seg75}.
\end{rem}
	
	The $\mathcal{R}$-module structure on $\mathbb{F}_2[u]$ is determined by the equation \[Q^r u = \begin{cases} u^2, & r = 2 \\ 0, & \text{otherwise}.\end{cases}\]  
	
	Application of the Cartan formula shows that for $n > 0$, \[Q^r u^n = \begin{cases} u^{2n}, & r= 2n \\ 0, & \text{otherwise}, \end{cases}\] so for $I$ allowable we have $Q^I u^n = 0$ unless $I = (2^s n, \ldots, 4n, 2n)$.  In particular, $e(I) = 2n$.  It follows that \[\Ann_{B(d+1)} \mathbb{F}_2[u] \cong \mathbb{F}_2 \{u^n \mid 2n > d\},\] and thus \[K^*(\mathbb{F}_2[u]^\vee, \mathcal{U}, \mathbb{F}_2) \cong \prod_{n \geq 1} u^n \otimes L(0)_{<2n}.\]  The differential in this unstable co-Koszul complex is given by \[d(u^n \Sq^{I+1}) = \sum_{r \geq 0} (Q^r u^n) \Sq^{r+1, I+1} = u^{2n} \Sq^{2n+1, I+1}\] for $I + 1$ admissible and $2n > \abs{I}$.   The verification of the fact $d^2 = 0$ uses the easy identity $\Sq^{2r+1} \Sq^{r+1} = 0$.  Therefore, we have
	
\begin{prop}
	\label{prop:E2GHM}
	The $E_2$-page of the Goerss-Hopkins-Miller spectral sequence computing $\mathbb{G}_m(H\mathbb{F}_2[u])$ is \[\prod_{n \text{ odd}} u^n \otimes \ker(\Sq^{2n+1}) \times \prod_{n \text{ even}} u^n \otimes \ker(\Sq^{2n+1})/\im(\Sq^{n+1}).\]  
\end{prop}

	Before continuing, let us be more concrete and analyze the strict units of the truncated polynomial rings $H\mathbb{F}_2[u]/u^{n+1}$.  For instance, let $n = 5$.  The $E_1$-page of the spectral sequence computing $\mathbb{G}_m(H\mathbb{F}_2[u]/u^6)$ is shown in Figure \ref{fig:GHMSS}.  
	
\begin{sseqdata}[
	name=GHMSS1, Adams grading, 
	classes=fill,
	class pattern=linear, class placement transform={rotate=90},
	class labels={left}, label distance=1pt,
	axes type=center, left clip padding=10mm,
	font=\tiny,
	xscale=1.2,
	yscale=1.5
	]
	\class["u"](2,0)
	\class["u \Sq^2"](0,1)
	\class["u \Sq^3"](-1,1)

	\class["u^2"](4,0)
	\class["u^2 \Sq^2"](2,1)
	\class["u^2 \Sq^3"](1,1)
	\class["u^2 \Sq^4"](0,1)
	\class["u^2 \Sq^5"](-1,1)

	\class["u^3"](6,0)
	\class["u^3 \Sq^2"](4,1)
	\class["u^3 \Sq^3"](3,1)
	\class["u^3 \Sq^4"](2,1)
	\class["u^3 \Sq^5"](1,1)
	\class["u^3 \Sq^6"](0,1)
	\class["u^3 \Sq^7"](-1,1)
		
	\class["u^4"](8,0)
	\class["u^4 \Sq^2"](6,1)
	\class["u^4 \Sq^3"](5,1)
	\class["u^4 \Sq^4"](4,1)
	\class["u^4 \Sq^5"](3,1)
	\class["u^4 \Sq^6"](2,1)
	\class["u^4 \Sq^{4,2}"](2,2)
	\class["u^4 \Sq^7"](1,1)
	\class["u^4 \Sq^{5,2}"](1,2)
	\class["u^4 \Sq^8"](0,1)
	\class["u^4 \Sq^{6,2}"](0,2)
	\class["u^4 \Sq^9"](-1,1)
	\class["u^4 \Sq^{7,2}"](-1,2)
	\class["u^4 \Sq^{6,3}"](-1,2)
	
	\class["u^5"](10,0)
	\class["u^5 \Sq^2"](8,1)
	\class["u^5 \Sq^3"](7,1)
	\class["u^5 \Sq^4"](6,1)
	\class["u^5 \Sq^5"](5,1)
	\class["u^5 \Sq^6"](4,1)
	\class["u^5 \Sq^{4,2}"](4,2)
	\class["u^5 \Sq^7"](3,1)
	\class["u^5 \Sq^{5,2}"](3,2)
	\class["u^5 \Sq^8"](2,1)
	\class["u^5 \Sq^{6,2}"](2,2)
	\class["u^5 \Sq^9"](1,1)
	\class["u^5 \Sq^{7,2}"](1,2)
	\class["u^5 \Sq^{6,3}"](1,2)
	\class["u^5 \Sq^{10}"](0,1)
	\class["u^5 \Sq^{8,2}"](0,2)
	\class["u^5 \Sq^{7,3}"](0,2)
	\class["u^5 \Sq^{11}"](-1,1)
	\class["u^5 \Sq^{9,2}"](-1,2)
	\class["u^5 \Sq^{8,3}"](-1,2)
	
	\d1(2,0,1,1)
	\d1(4,0,1,2)
	\d1(2,1,1,1)
	\d1(0,1,2,2)
\end{sseqdata}

\begin{figure}
	\caption{The ``$E_1$-page'' of the Goerss-Hopkins spectral sequence using Koszul resolutions computing the strict units of the truncated polynomial ring $H\mathbb{F}_2[u]/u^6$.}
	\label{fig:GHMSS}
	\printpage[name=GHMSS1, page=1]
\end{figure}

	Here are some remarks.  
\begin{rems}
\label{rem:GHSS}
\mbox{}
\begin{enumerate}[(a)]
	\item In principle, we can compute the entire $E_2$-page using this method -- all that is required are manipulations with the Adem relations.  However, we also note that the co-Koszul complex is often still very, very large.  
	
	\item For smallish examples (in terms of $n$) such as the above, there are no more differentials possible on the $E_2$-page, so the spectral sequence collapses and we can deduce the homotopy groups of the space of strict units.  See theorem \ref{thm:table}.
	
	\item The $E_1$-page splits as a complex: the terms involving $\{u, u^2, u^4, u^8, \ldots\}$ form one complex, the terms involving $\{u^3, u^6, u^{12}, u^{24}, \ldots\}$ form a second complex, the terms involving $\{u^5, u^{10}, u^{20}, u^{40}, \ldots\}$ form another, and so on.  
	
	In fact, this can be promoted to a splitting of spectral sequences and in fact of spectra, by a result of Steiner's which we recall as theorem \ref{thm:Steiner}.  
	
\end{enumerate}
\end{rems}

	Let us return to the spectral sequence computing the strict units of the untruncated polynomial ring $H\mathbb{F}_2[u]$.  It seems difficult to give a compact explicit description of the $E_2$-page, but it appears to be relatively sparse, despite the impression from the example above.  
	
	First let us compute the $1$-line of the spectral sequence.  
\begin{prop}
	$\Ext_\mathcal{U}^{1,*}(\mathbb{F}_2, \mathbb{F}_2[u])$ is spanned by $\{u^n \Sq^{n+1} \mid n \text{ odd}\}$.  
\end{prop}
\begin{proof}
	A general element in the complex in filtration level $s = 1$ has the form $m = \sum_j u^{n_j} \Sq^{i_j + 1}$ with $2n_j > i_j$.  If $m$ is a cocycle, then by degree reasons we may assume without loss of generality that $m = u^n \Sq^{i+1}$ with $2n > i$.  
	
	We have \[d(u^n \Sq^{i+1}) = u^{2n} \Sq^{2n+1,i+1}.\]  If $i < n$, then $\Sq^{2n+1, i+1}$ is admissible, and in particular nonzero.  So $u^n \Sq^{i+1}$ is not a cocycle.  
	
	Now suppose $n \leq i < 2n$.  Then we can apply the Adem relation \[\Sq^{2n+1} \Sq^{i+1} = \sum_{t = 2n - i + 1}^n \binom{i-t}{2n-2t+1} \Sq^{2n _i + 2 - t} \Sq^t.\]  Note that all the terms $\Sq^{2n+i+2-t, t}$ are admissible, so no further Adem relations need to be applied.  Thus we see that $u^n \Sq^{i+1}$ is a cocycle iff $\binom{i-t}{2n+1-2t} \equiv 0 \pmod{2}$ for all $2n - i + 1 \leq t \leq n$.  If $i = n$, then the binomial coefficient $\binom{n-t}{2(n-t) + 1} = 0$ for all $t$, so we obtain some cocycles in this way.  However, if additionally $n$ is even, then $d(u^{n/2}) = u^n \otimes \Sq^{n+1}$, so these cocycles obtained in this way are also coboundaries.  
	
	We claim that there are no other cocycles for $s = 1$.  Given $n < i < 2n$, set $t = 2n - i + 1$.  The binomial coefficient for this value of $t$ is \[\binom{i - (2n-i+1)}{2n+1 - 2(2n - i + 1)} = \binom{2i-2n-1}{2i-2n-1} = 1.\]  Thus we conclude that $\Ext_\mathcal{U}^1(\mathbb{F}_2, \mathbb{F}_2[u])$ is spanned by the classes $\{u^n \Sq^{n+1} \mid n \text{ odd}\}$. 
\end{proof}
	
	Recall from (\ref{prop:E2GHM}) that the $E_2$-page of the spectral sequence is \[\prod_{n \text{ odd}} u^n \otimes \ker(\Sq^{2n+1}) \times \prod_{n \text{ even}} u^n \otimes \ker(\Sq^{2n+1})/\im(\Sq^{n+1}).\]  The previous proposition shows that there is \emph{no} contribution to the $s = 1$ line from the second factor $\ker(\Sq^{2n+1})/\im(\Sq^{n+1})$.  It thus behooves us to study this quotient.  
	
	It seems that this quotient is sparse at least in low degrees:
\begin{prop}
	For $n \leq 15$, the nonzero homology classes in $(\ker \Sq^{2n+1} / \im \Sq^{n+1})|_{L(0)_{<2n}}$ are:
	\begin{enumerate}[(i)]
		\item $\Sq^{8,4,2}$ in $\ker \Sq^{13}/\im \Sq^7$,
		\item $\Sq^{12,6,3}$ in $\ker \Sq^{21}/\im \Sq^{11}$, and
		\item $\langle \Sq^{16,8,4}, \Sq^{16,8,4,2}, \Sq^{17,8,4,2} \rangle$ in $\ker \Sq^{29}/\im \Sq^{15}$.  
	\end{enumerate}
\end{prop}
\begin{proof}
	Computer search.
\end{proof}

	From this data, one might conjecture that $\Sq^{4i,2i,i}$ is always a nontrivial class in the quotient.  This is indeed the case.  
\begin{prop}
	For $i \geq 1$, the element $\Sq^{4i,2i,i}$ is nonzero in the subquotient $\ker \Sq^{8i-3} / \im \Sq^{4i-1}$.   
\end{prop}

\begin{rem}
	The degree of $u^{4i-2} \Sq^{4i,2i,i}$ is $i - 4$, so these elements on the $s = 3$ line don't contribute to the quadrant of interest in the spectral sequence until $i \geq 4$.  
\end{rem}

\begin{proof}[Proof of proposition]
	First we show that $\Sq^{4i,2i,i}$ is a cocycle, i.e., $\Sq^{8i-3} \Sq^{4i,2i,i} = 0$.  We shall need the following relations in the Steenrod algebra:
	\begin{itemize}
		\item $\Sq^{8i-3} \Sq^{4i} + \Sq^{8i-1} \Sq^{4i-2} = 0$
		\item $\Sq^{4i-2} \Sq^{2i} + \Sq^{4i-1} \Sq^{2i-1} = 0$
		\item $\Sq^{2i-1} \Sq^i = 0$.
	\end{itemize}
	
	These formulas are special cases of the more general equation \[\Sq^{2^k i - (2^{k-2} + 1)} \Sq^{2^{k-1} i} + \Sq^{2^k i - 1} \Sq^{2^{k-1} i - 2^{k-2}}\] for $k \geq 2$, which can be obtained by stripping $(2^{k-2})$ from the more obvious relation $\Sq^{2^k i - 1} \Sq^{2^{k-1} i} = 0$.  
	
	Now, we have \[\Sq^{8i-3} \Sq^{4i} \Sq^{2i} \Sq^i = \Sq^{8i-1} \Sq^{4i-2} \Sq^{2i} \Sq^i = \Sq^{8i-1} \Sq^{4i-1} \Sq^{2i-1} \Sq^i = 0.\] 
	
	It remains to show that $\Sq^{4i,2i,i}$ is not in the image of $\Sq^{4i-1}$.  By comparing the degree and length of the source and target, we are reduced to considering the image of $\Sq^{4i-1}$ applied to the subspace spanned by the elements $\Sq^{3i+1-j,j}$ where $1 \leq j < i$.  
	
	We claim that $\Sq^{4i,2i,i}$ does not even show up as one of the terms in any $\Sq^{4i-1} \Sq^{3i+1-j,j}$.  First observe that we can apply an Adem relation to the first two factors: \[\Sq^{4i-1} \Sq^{3i+1-j} = \sum_t \binom{3i-j-t}{4i-2t-1} \Sq^{7i-j-t} \Sq^t.\]  These terms are all admissible, so we compare the indices: $\Sq^{7i-j-t,t,j} = \Sq^{4i,2i,i}$.  We find that $j = i$ and $t = 2i$, and the binomial coefficient in this case is $0$.  So the term $\Sq^{4i,2i,i}$ does not appear in the image as claimed.  
\end{proof}

	It turns out that the subquotients $\ker \Sq^{2n+1} / \im \Sq^{n+1}$ control a large part of the homotopy groups of strict units of polynomial rings and truncated polynomial rings.  We will give a more precise conjecture about them later (\ref{conj:AdemHomology}). 

\section{The Postnikov tower of \texorpdfstring{$gl_1 H\mathbb{F}_2[u]$}{gl1(HF2[u])}}
\label{sec:Postnikov}
	Knowledge of the Postnikov $k$-invariants of $gl_1 R$ can be useful for computing maps into $gl_1 R$.  In this section, we describe the Postnikov tower of $gl_1 H\mathbb{F}_2[u]$ building on results in the literature, and use it to give another spectral sequence for computing $\mathbb{G}_m(H\mathbb{F}_2[u])$.  
	
	The first thing to observe that if $R$ is an Eilenberg-Mac~Lane spectrum of a ordinary graded commutative ring, then additively $R$ is a product of shifts of Eilenberg-Mac~Lane spectra, so all the $k$-invariants of $R$ as a spectrum are zero.  However, the $k$-invariants of $gl_1 R$ could still be very interesting as we shall see.  
	
	Let $H = H\mathbb{F}_2$ again.  
	
\subsection{The Postnikov tower}
	We first recall a result of Lawson and Mathew-Stojanoska on the $k$-invariant in $gl_1 H[u]/u^3$, where $\deg u = d$.  
	
\begin{prop}[{\cite[Prop.~5.2.2]{MS16}}]
\label{prop:BottomKInvt}
	The bottom $k$-invariant of $gl_1 H[u]/u^3$ is $\Sq^{d+1}$.  
\end{prop}


\begin{quest}
	In general, how are the $k$-invariants for $gl_1 R$ related to the $k$-invariants of $R$?  \cite{Hes15} has some interesting speculation about this problem using topological Andr\'{e}-Quillen cohomology.  
\end{quest}

	The other result we need concerns a splitting of $gl_1 H[u]$.  From now on let $\deg u = 2$ again.  In ordinary algebra, the units of the power series ring $\mathbb{F}_2 \llbracket u \rrbracket$ (with constant term $1$) are the (big) Witt vectors $W(\mathbb{F}_2)$, and it is well-known that the Witt vectors decomposes as a product.  The same holds in topology:
	
\begin{thm}[Steiner \cite{Ste79}]
\label{thm:Steiner}
	The spectrum $gl_1 H[u]$ decomposes: \[gl_1 H[u] \simeq \prod_{k \text{ odd}} g_k,\] where \[\pi_i(g_k) \cong \begin{cases} \mathbb{F}_2, & \frac{i}{2^{\nu_2(i)}} = k \\ 0, & \text{otherwise}. \end{cases}\]
\end{thm}

\begin{rem}
	Hess \cite{Hes19} has given an alternative explanation of these facts by analyzing the operations on the homology of the infinite loop space $GL_1 R$ for an $E_\infty$-$H$-algebra $R$.  Specifically, he recovered the nontriviality of the $k$-invariant of $gl_1 H[u]$ in \cite{MS16}, and extended this to odd primes.  See section \ref{sec:oddkinvt}.  His results also imply a truncated version of Steiner's splitting.  
\end{rem}

\begin{rem}
	While we won't need to rely on his results, Kraines \cite{Kra70, Kra73b} has also studied the Postnikov towers for the space $GL_1 H[u]$ as well as $BGL_1 H[u]$.  
\end{rem}

	As a consequence of theorem \ref{thm:Steiner}, we are reduced from studying $gl_1 H[u]$ to studying these spectra $g_k$.  The problems we encounter are roughly similar for different $k$, so we'll focus on $g_1$ for concreteness.  
	
	Using proposition \ref{prop:BottomKInvt}, we deduce the following.  
\begin{prop}
	\label{prop:g1Postnikov}
	The Postnikov tower for $g_1$ has the form:  
	\begin{center}
	\begin{tikzcd}
		& g_1 \ar[d] \\
		& \vdots \ar[d] \\
		\Sigma^{2^i} H \ar[r] \ar[rr, bend left, near end, "\Sq^{2^i + 1}"] & g_1/u^{2^i} \ar[r, "k_{2^{i+1} - 1}"'] \ar[d] & \Sigma^{2^{i+1} + 1} H \\
		& \vdots \ar[d] \\
		\Sigma^8 H \ar[r] \ar[rr, bend left, near end, "\Sq^9"] & g_1/u^8 \ar[r, "k_{15}"] \ar[d] & \Sigma^{17} H \\
		\Sigma^4 H \ar[r] \ar[rr, bend left, near end, "\Sq^5"] & g_1/u^4 \ar[r, "k_7"] \ar[d] & \Sigma^9 H \\
		& \Sigma^2 H \ar[r, "\Sq^3"] & \Sigma^5 H
	\end{tikzcd}
	\end{center}
\end{prop}

\begin{rem}
	The curved arrows give the first differential in the Atiyah-Hirzebruch spectral sequence described in the next subsection.  It is instructive to also compare them with the differential obtained in (\ref{thm:UnstableCokoszul}).  
\end{rem}

\begin{proof}
	Apply naturality with the map $gl_1 H[v]/v^3 \to g_1/u^{2^{n+1}}$ induced by $v \mapsto u^{2^{n-1}}$, so $\deg v = 2^n$.  We have a map of fiber sequences
	\begin{center}
	\begin{tikzcd}
		gl_1 H[v]/v^3 \ar[r] \ar[d] & gl_1 H[v]/v^2 \ar[r, "\Sq^{2^n+1}"] \ar[d] & \Sigma^{2^{n+1} + 1} H \ar[d, equal] \\
		g_1/u^{2^{n+1}} \ar[r] & g_1/u^{2^n} \ar[r, "k_{2^{n+1}-1}"'] & \Sigma^{2^{n+1} + 1} H.
	\end{tikzcd}
	\end{center}
	
	We know the $k$-invariant in the top row by proposition (\ref{prop:BottomKInvt}).  Therefore the map \[\Sigma^{2^n} H \to g_1/u^{2^n} \xrightarrow{k_{2^{n+1}-1}} \Sigma^{2^{n+1} + 1} H\] is given by $\Sq^{2^n + 1}$.  
\end{proof}

\subsection{The Postnikov spectral sequence}

	We analyze the spectral sequence computing $\Map(H\mathbb{Z},g_1)$ associated to the tower obtained by mapping $H\mathbb{Z}$ into (\ref{prop:g1Postnikov}).  To get rid of clutter, we restrict to the analogous tower for $g_1/u^8 = \tau_{< 16} g_1$.  

\begin{prop}
	\label{prop:g1AHSS}
\begin{sseqdata}[
	name=g1u8, Adams grading, 
	classes=fill,
	class pattern=linear, class placement transform={rotate=45},
	class labels={left}, label distance=1pt,
	axes type=center, left clip padding=6mm,
	font=\tiny,
	xscale=1.5
	]
	\class["u"](2,0)
	\class["u \Sq^2"](0,0)
	\class["u \Sq^3"](-1,0)

	\class["u^2"](4,1)
	\class["u^2 \Sq^2"](2,1)
	\class["u^2 \Sq^3"](1,1)
	\class["u^2 \Sq^4"](0,1)
	\class["u^2 \Sq^5"](-1,1)

	\class["u^4"](8,2)
	\class["u^4 \Sq^2"](6,2)
	\class["u^4 \Sq^3"](5,2)
	\class["u^4 \Sq^4"](4,2)
	\class["u^4 \Sq^5"](3,2)
	\class["u^4 \Sq^6"](2,2)
	\class["u^4 \Sq^{4,2}"](2,2)
	\class["u^4 \Sq^7"](1,2)
	\class["u^4 \Sq^{5,2}"](1,2)
	\class["u^4 \Sq^8"](0,2)
	\class["u^4 \Sq^{6,2}"](0,2)
	\class["u^4 \Sq^9"](-1,2)
	\class["u^4 \Sq^{7,2}"](-1,2)
	\class["u^4 \Sq^{6,3}"](-1,2)
	
	\d1(2,0)
	\d1(4,1)
	\d1(2,1,,2)
	\d1(0,1,,2)
	
	
\end{sseqdata}

	The $E_1$-page is:
	
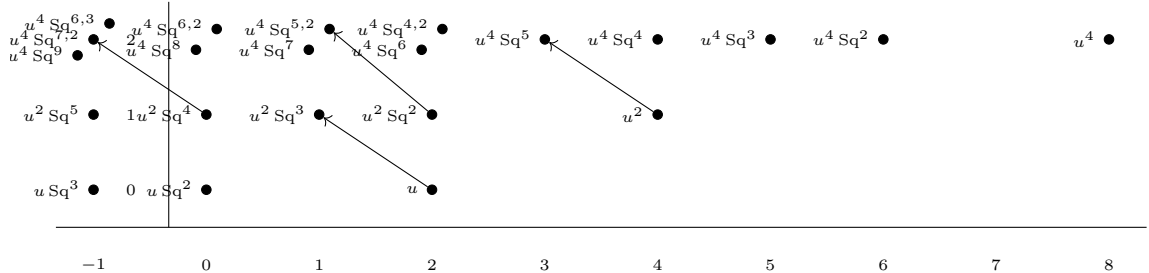
\printpage[name=g1u8, page=1]

	The $E_2$-page is:
	
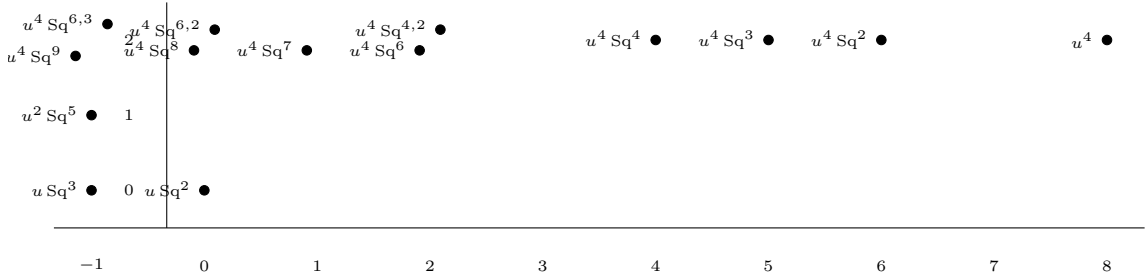
\printpage[name=g1u8, page=2]

	(Note, these charts are accurate in the range $t - s \geq 0$.)
\end{prop}
\begin{proof}
	The $d_1$-differentials can be read off from the Postnikov tower.  It remains to show that the class $u\Sq^2$ survives the $E_2$-page, so that it is a permanent cycle in the spectral sequence.  The quotient map $q: g_1 \to g_1/u^8$ induces a map of spectral sequences $q_*$.  We have \[d_2(u\Sq^2) = d_2(q_* u\Sq^2) = q_*(d_2 u\Sq^2) = q_*(0) = 0.\]
\end{proof}

\begin{rem}
	In general, using naturality to compare the spectral sequences for mapping $H\mathbb{Z}$ into the units of various truncations of polynomials seems to be an effective tool for determining longer differentials. 
\end{rem}

\begin{rems}
\mbox{}
\begin{enumerate}[(a)]
	\item The abundance of classes contributed from $u^4$ in this example is an artifact of restricting to $g_1/u^8$, since the $d_1$-differential on these classes have zero target.  
	\item	As before, if $n$ is small, there will be no more room for differentials, and we get a complete description of $\pi_* \mathbb{G}_m(H[u]/u^{n+1})$.  
	\item However, compared to the spectral sequence of (\ref{prop:E2GHM}), the filtrations of the elements are different; in fact, the length filtration coming from the Koszul resolution and the Postnikov filtration are \emph{not} compatible, i.e., neither refines the other.  One way to think about these filtrations is in terms of the cellular filtration of $H\mathbb{Z}$.  In the symmetric product of spheres filtration, instead of attaching cells individually to build $H\mathbb{Z}$, we instead attach entire groups of cells at once in the form of Steinberg summands.   On the other hand, it is well-known that the Postnikov filtration of the target gives rise to the Atiyah-Hirzebruch spectral sequence for mapping spaces, which can also be constructed using the cellular filtration of the source \cite{Mau63}.  However, here we have used Steiner's splitting of the spectrum of units $gl_1 H[u]$ to ``accelerate'' the Postnikov filtration in a different way.  
\end{enumerate}
\end{rems}

\subsection{Extension problems}

	Now that we have a description of the $E_\infty$-page of the spectral sequence in certain situations, it remains to check whether or not there are extensions.  Each dot in our spectral sequence represents a $\mathbb{F}_2$; the question is whether or not there could be multiplication-by-$2$ extensions amongst these classes.  
	
	Let $R = H[u]$ or $H[u]/u^{n+1}$.  Consider a class $x: \Sigma^i H\mathbb{Z} \to gl_1 R$.  We may check whether or not it is $2$-torsion using the cofiber sequence \[H\mathbb{Z} \xrightarrow{2} H\mathbb{Z} \to H\mathbb{F}_2\] and determining whether $x$ is in the image of $[\Sigma^i H\mathbb{F}_2, gl_1 R]$.  In practice, this means computing the analogous spectral sequences with $H\mathbb{F}_2$ in place of $H\mathbb{Z}$.  We note that there is an injection on from the $E_1$-page of the spectral sequence computing $\Map(H\mathbb{Z}, gl_1 R)$ to the $E_1$-page of the spectral sequence computing $\Map(H\mathbb{F}_2, gl_1 R)$.  The question is whether or not the classes corresponding to the classes in the spectral sequence for $\Map(H\mathbb{Z}, gl_1 R)$ survive the spectral sequence for $\Map(H\mathbb{F}_2, gl_1 R)$. In most cases they do:
	
\begin{prop}
	For $n \leq 8$, $\pi_* \mathbb{G}_m(H[u])$ is $2$-torsion, i.e., there are no additive extensions.  
\end{prop}

\begin{rem}
	It does not seem difficult at all to relax the restriction that $n \leq 8$ a little further -- it is just a matter of doing the tedious, but entirely routine, computations.  
\end{rem}

\subsection{Table of homotopy groups of $\mathbb{G}_m(H[u]/u^{n+1})$}
	Here is a summary for the homotopy groups of the strict units for the truncated polynomial rings $H[u]/u^{n+1}$ for $n \leq 8$.  

\begin{thm}
	\label{thm:table}
	The groups $\pi_i H[u]/u^{n+1}$ are 2-torsion with the following ranks:
\begin{equation*}	
\begin{array}{c|cccccccc}
	\hline
	n & 1 & 2 & 3 & 4 & 5 & 6 & 7 & 8 \\
	\hline
	\dim \pi_0 & 1 & 1 & 3 & 5 & 8 & 9 & 14 & 17 \\
	\dim \pi_1 & 0 & 0 & 1 & 2 & 5 & 7 & 11 & 14 \\
	\dim \pi_2 & 1 & 1 & 2 & 3 & 5 & 7 & 11 & 14 \\
	\dim \pi_3 & & 0 & 1 & 1 & 3 & 4 & 7 & 10 \\
	\dim \pi_4 & & 1 & 2 & 2 & 4 & 5 & 8 & 10 \\
	\dim \pi_5 & & & 0 & 1 & 2 & 3 & 6 & 7 \\
	\dim \pi_6 & & & 1 & 2 & 3 & 4 & 6 & 8 \\
	\dim \pi_7 & & & & 0 & 1 & 2 & 4 & 6 \\
	\dim \pi_8 & & & & 1 & 2 & 3 & 5 & 6 \\
	\dim \pi_9 & & & & & 0 & 1 & 2 & 4 \\
	\dim \pi_{10} & & & & & 1 & 2 & 3 & 5 \\
	\dim \pi_{11} & & & & & & 0 & 1 & 2 \\
	\dim \pi_{12} & & & & & & 1 & 2 & 3 \\
	\dim \pi_{13} & & & & & & & 0 & 1 \\
	\dim \pi_{14} & & & & & & & 1 & 2 \\
	\dim \pi_{15} & & & & & & & & 0 \\
	\dim \pi_{16} & & & & & & & & 1 \\
	\hline
\end{array}
\end{equation*}
\end{thm}

\begin{rem}
	Again, it is not too hard to compute a bit more than shown here before running into a potential differential that needs resolving by other means.    
\end{rem}

\subsection{A conjecture}
	In the mod 2 Steenrod algebra, we have the fundamental relation \[\Sq^{2n+1} \Sq^{n+1} = 0.\]  It turns out that all the Adem relations can be deduced from these simple relations using the coaction of the dual Steenrod algebra, i.e., by a stripping process \cite{Kri65}.  
	
	Since the composition of $\Sq^{n+1}$ and $\Sq^{2n+1}$ is null, it makes sense to ask for the ``homology'' of this relation, i.e., $\ker \Sq^{2n+1}/\im \Sq^{n+1}$ in any left $\mathcal{A}$-module such as $H^* H\mathbb{Z}$.  This is the basic calculation needed to compute the $E_2$-page of the spectral sequence for the strict units of the polynomial ring $H[u]$.  Given the results in this section, we are led to the following more precise conjecture.  
	
\begin{conj}
\label{conj:AdemHomology}
	In $\mathcal{A}/\mathcal{A}\beta$, we have $\ker \Sq^{8k+1} = \im \Sq^{4k+1}$ in degrees up to $8k+1$ for all $k \geq 0$. 
\end{conj}

	If the conjecture is true, then the individual spectral sequences for $\Map(H\mathbb{Z},g_k)$ will be concentrated in the first two rows from the $E_2$-page onwards, and thus collapse.  This implies that we would have \[\pi_* \mathbb{G}_m(H[u]) \cong \gen{u^k \otimes \ker \Sq^{2k+1} \oplus u^{2k} \otimes \ker \Sq^{4k+1}/\im \Sq^{2k+1} \mid k \text{ odd}}\] possibly up to extension problems.

\begin{rem}
	This conjecture has been verified by a computer for small values of $k$.  
\end{rem}

\subsection{Odd-primary $k$-invariants}
\label{sec:oddkinvt}

	We establish an odd primary analogue of proposition \ref{prop:BottomKInvt} about the $k$-invariants of unit spectra, which may be useful for investigations at other primes.  (See problem 1.9.9 in \cite{Law20}.)  If $p$ is odd, then $\deg u = 2i$ has to be even for the polynomial ring $H\mathbb{F}_p[u]$ to make sense.  
	
\begin{thm}
\label{thm:OddKInvt}
	Let $p$ be odd.  The first nontrivial $k$-invariant of $gl_1 H\mathbb{F}_p[u]$, where $\deg u = 2i$, is (up to a unit) \[k_{2pi-1} = \beta P^i: \Sigma^{2i} H\mathbb{F}_p \to \Sigma^{2pi + 1} H\mathbb{F}_p.\]  
\end{thm}
\begin{proof}
	Let us first explain what the statement means.  The result in theorem \ref{thm:Steiner} actually applies in greater generality. In particular, it applies at odd primes too, and $gl_1 H\mathbb{F}_p[u]$ splits as a product: \[gl_1 H\mathbb{F}_p[u] \simeq H\mathbb{F}_p^\times \times \prod_{\substack{k \geq 1 \\ (k,p) = 1}} g_k,\] where \[\pi_i(g_k) \cong \begin{cases} \mathbb{F}_p, & \frac{i}{2p^{\nu_p(i)}} = k \\ 0, & \text{otherwise}. \end{cases}\]  This means that the first few $k$-invariants of $gl_1 H\mathbb{F}_p[u]$ are zero, and the first possibly nontrivial $k$-invariant is in fact $k_{2pi-1}: \Sigma^{2i} H\mathbb{F}_p \to \Sigma^{2pi+1} H\mathbb{F}_p$.  
	
	To determine the map, we compare the Goerss-Hopkins-Miller spectral sequence in theorem \ref{thm:OddPCokoszul} and the modified Postnikov spectral sequence constructed in this section.  On the one hand, the first nontrivial $k$-invariant of a spectrum gives the first nonzero differential in the Postnikov spectral sequence.  On the other hand, the description in theorem \ref{thm:OddPCokoszul} tells us that if $\deg u = 2i$, then the differential is given by \[d(u \otimes 1) = -u^p \otimes \beta P^i.\]  We conclude that the $k$-invariant $k_{2pi-1}$ is given by $\beta P^i$ up to a unit.  
\end{proof}

\begin{rem}
	The proof in \cite{MS16} can essentially be adapted to this case too.  Once we know that the first possibly nontrivial $k$-invariant is $k_{2pi-1}$, the fact that $\iota^p = Q^i \iota = 0$ for the generator $\iota \in H_{2i} K(\mathbb{F}_{2i})$ implies that $k_{2pi-1}$ is not zero.  But $\Omega^\infty k_{2pi-1} = 0$, which means that $k_{2pi-1}$ must have sufficiently large excess in the mod $p$ Steenrod algebra, and the only candidate in the appropriate degree is $\beta P^i$.  
\end{rem}

\begin{rem}
	This also follows from a result recently obtained by Hess \cite{Hes19}, who showed that if $R$ is any $E_\infty$-$H\mathbb{F}_p$-algebra, then \[\tau_{[n, pn-1]} R \simeq \tau_{[n, pn-1]} gl_1 R\] for $n \geq 1$.  Consequently, if $R = H\mathbb{F}_p[u]$, $\deg u = 2i$, then this result shows that all the $k$-invariants for $gl_1 R$ up to but not including $k_{2pi-1}$ vanish since they do so additively.  The above theorem identifies the next $k$-invariant, which is nonzero and witnesses that the truncation range cannot be improved in general.  
\end{rem}

\section{The symmetric product of spheres filtration and transfers}
\label{sec:SymProd}

\subsection{Transfers and the cohomology of $H\mathbb{Z}$}
	The essential upshot of Priddy's and Miller's Koszul resolutions is that they allow one to filter the algebra of power operations that appears in the Goerss-Hopkins-Miller spectral sequence by word length.  This length filtration can be realized topologically too, as we explain in this section.  This gives an alternative way to frame and approach the calculation of strict units even for $E_\infty$-rings whose power operations are not yet completely understood.  In particular, in this section we use this new perspective to say something about the strict units of the sphere spectrum (\ref{thm:LsGL1S0}, \ref{thm:MsGL1S0}).  Along the way we shall also prove other results which may be of independent interest.  

	The starting point is the theorem of Dold-Thom, which asserts that the ``designer'' spectrum $H\mathbb{Z}$ is built out of ``space-like'' spectra in an explicit way.  This is useful, because while maps from $H\mathbb{Z} \to gl_1 R$ can be difficult to compute, maps from spaces to $gl_1 R$ are easier.  Let $\SP^n(S^0)$ be the $n$-th symmetric power of the sphere spectrum, and let $\SP^\infty(S^0) := \colim_n \SP^n(S^0)$.  We have
\begin{thm}[Dold-Thom \cite{DT58}]
	$\SP^\infty(S^0) \simeq H\mathbb{Z}$.
\end{thm}
		
	The infinite symmetric power $\SP^\infty(S^0)$ admits a cardinality filtration by the $\SP^n(S^0)$'s.  From now on, we fix a prime $p$ and implicitly $p$-complete everything.  Nakaoka showed that the symmetric power filtration has the property that if $n$ is not a $p$-th power, then the map \[\SP^{n-1}(S^0) \to \SP^n(S^0)\] is $p$-locally a weak equivalence.  Furthermore, we have a sequence of cofibrations
	\begin{center}
	\begin{tikzcd}
		S^0 \simeq \SP^1(S^0) \ar[r] & \SP^p(S^0) \ar[r] \ar[d] & \SP^{p^2}(S^0) \ar[r] \ar[d] & \cdots \ar[r] & \SP^\infty(S^0) \simeq H\mathbb{Z} \\
		& \Sigma L(1) & \Sigma^2 L(2)
	\end{tikzcd}
	\end{center}
	
	\noindent where the indicated cofibers of the horizontal maps are suspensions of Steinberg summands \cite{MP83}.
	
	For any spectrum $E$, applying mapping spaces $\Map(-, E)$ to this sequence yields a tower of fibrations
	\begin{center}
	\begin{tikzcd}
		\Map(S^0, E) \ar[dr, dashed] & \Map(\SP^{p^1}(S^0), E) \ar[l] \ar[dr, dashed] & \Map(\SP^{p^2}(S^0), E) \ar[l] \ar[dr, dashed] & \cdots \ar[l] & \Map(H\mathbb{Z}, E). \ar[l] \\
		& \Map(\Sigma L(1), E) \ar[u] & \Map(\Sigma^2 L(2), E) \ar[u] & \cdots
	\end{tikzcd}
	\end{center}
	
	This gives a spectral sequence.
	
\begin{prop}
	\label{prop:SymFiltSS}
	There is a spectral sequence with signature \[E_1^{s,t} = [\Sigma^t L(s), E] \Rightarrow \pi_{t-s} \Map(H\mathbb{Z}, E).\] Moreover, the $E_1$-differential \[d_1: \pi_t \Map(L(s), E) \to \pi_t \Map(L(s+1), E)\] is induced by the transfer maps $\delta_s: L(s+1) \to L(s)$ in $E$-cohomology.  
\end{prop}
		
	The sequence \[\cdots \xrightarrow{\delta_2} L(2) \xrightarrow{\delta_1} L(1) \xrightarrow{\delta_0} L(0) \xrightarrow{\delta_{-1}} H\mathbb{Z}\] has been studied in a variety of settings with important implications for homotopy theory.  We recall several instances below.  
	
\begin{enumerate}[(a)]
	\item Nakaoka \cite{Nak58} studied this sequence in mod $p$ cohomology, and showed that the symmetric product filtration realizes the length filtration on $H\mathbb{F}_p^*(H\mathbb{Z}) \cong \mathcal{A}/\mathcal{A} \beta$.  So the maps $\delta_s$ induce the zero map in ordinary cohomology.  
	
	\item Kuhn \cite{Kuh82} and Kuhn-Priddy \cite{KP85} studied this sequence in $p$-local homotopy, and showed that the sequence is exact on $p$-local homotopy groups.  This is the celebrated \emph{Whitehead conjecture} \cite{Mil71}.  
	
	\item Cathcart \cite{Cat88} and Stroilova \cite{Str12} studied this sequence in $p$-local \emph{cohomotopy}, and showed that the induced sequence is also exact on cohomotopy groups.  In fact, this was done via establishing a splitting of the Spanier-Whitehead duals of Steinberg summands \[\mathbb{D}L(s) \simeq L(s) \vee L(s-1).\]  
	
	Plugging this result into (\ref{prop:SymFiltSS}), this gives a rather roundabout proof of the fact that $\Map(H\mathbb{Z}, S^0) \simeq *$.  
\end{enumerate}

	Understanding the $d_1$ differential in the spectral sequence of (\ref{prop:SymFiltSS}) amounts to studying the behavior of the sequence in $E$-cohomology, so in this section we shall prove several ``versions'' of Whitehead conjecture before tackling the case $E = gl_1 S^0$.  

	It is convenient to recall the group-theoretic interpretation of the Steinberg summands $L(s)$ and related spectra. By the Segal conjecture, now a theorem\footnote{The theorem was eventually proved by Carlsson \cite{Car84} building on earlier work, but we only require the version for elementary abelian groups, which was proved earlier by Adams-Gunawardena-Miller \cite{AGM85}.  We shall use the extension to maps between classifying spaces given by \cite{LMM82}.}, we know all the stable maps between the classifying spaces $BV$ where $V$ is a finite-dimensional vector space over the finite field $\mathbb{F}_p$ in terms of Burnside modules.  More explicitly, 

\begin{prop}[Special case of the Segal conjecture \cite{LMM82}]
	Let $V$ and $W$ be finite-dimensional vector spaces over $\mathbb{F}_p$.  Then,	\[[BV_+, BW_+] \cong A(V,W)^{\wedge}_{I(V)} \cong \bigoplus_{(V',f)} \mathbb{Z}_p \cdot u_{(V',f)},\] where $(V',f)$ ranges over all subspaces $V' \leq V$ and homomorphisms $f: V' \to W$.  
		
		The element $u_{(V',f)}$ represents the stable map \[BV_+ \xrightarrow{\tr^V_{V'}} BV'_+ \xrightarrow{Bf} BW_+.\] 	Composition among the $u_{(V',f)}$'s is determined by the double coset formula.  
\end{prop}
	
	As a consequence of the corollary, there is an inclusion \[\mathbb{Z}_p \widetilde{\Hom}_{\mathbb{F}_p}(V,W) \to [BV_+, BW_+], \quad f \mapsto u_{(V,f)},\] where $\mathbb{Z}_p \widetilde{\Hom}_{\mathbb{F}_p}(V,W)$ is the group ring on $\Hom_{\mathbb{F}_p}(V,W)$ with coefficients in $\mathbb{Z}_p$ after identifying the zero homomorphism with zero in the ring.  
	
	Now, let $V_s \cong \mathbb{F}_p^s$ be an $s$-dimensional vector space over $\mathbb{F}_p$.  Let $B_s$ be the Borel subgroup of upper triangular matrices in $GL(V_s)$, and $\Sigma_s \leq GL(V_s)$ the symmetric group of permutation matrices.  Also, let $U_s$ be the maximal unipotent subgroup of upper-triangular matrices with ones on the diagonal; it is a $p$-Sylow subgroup of $GL(V_s)$.  Define the \emph{Steinberg idempotent} \[e_s = \frac{1}{[GL(V_s): U_s]} \sum_{\substack{b \in B_s \\ \sigma \in \Sigma_s}} (-1)^\sigma b\sigma \in \mathbb{Z}_p[GL(V_s)].\]  One can show that $e_s^2 = e_s$.  
	
	By the remark following the previous corollary, $e_s$ acts on $BV_{s+}$ via stable maps, and determines a summand $M(s) := e_s \cdot BV_{s+}$ of $BV_{s+}$.  In fact, $BV_{s+}$ contains $p^{\binom{s}{2}}$ summands equivalent to $M(s)$ \cite{MP83}.  Here are some more facts about these spectra:
\begin{enumerate}[(i)]
	\item The spectra $M(s)$ are the associated graded pieces of a filtration of $H\mathbb{Z}/p$, analogous to how the $L(s)$ are the associated graded pieces of the symmetric product of spheres filtration of $H\mathbb{Z}$.  
	
	\item The spectrum $M(s)$ splits into $L(s) \vee L(s-1)$.  
\end{enumerate}

	As a result of these facts, especially (ii), we will generally work with the $M(s)$ spectra instead of the $L(s)$ spectra, keeping in mind that if we are in fact interested in only the $L(s)$ summand after all then we need only split off its contribution at the end.   
	
\begin{rem}
	The spectra $L(s)$ also have a more direct group-theoretic construction via Thom spectra.  Consider again the classifying space $BV_s$.  Let $\bar{\rho}_s$ be the reduced regular representation of $V_s$, thought of as a vector bundle on $BV_s$.  The associated Thom spectrum $BV_s^{\bar{\rho}_s}$ still admits an action by $\mathbb{Z}_p[GL(V_s)]$, and we have \[L(s) \simeq e_s \cdot BV_s^{\bar{\rho}_s}.\]  
\end{rem}

	One can also completely calculate the maps between the $M(s)$'s (or the $L(s)$'s) using either group theory \cite{Nis87} or Adams spectral sequences \cite{Cat88}.  We shall need the following in the sequel.  
	
\begin{prop}
	The map $\delta_s: M(s+1) \to M(s)$ (resp.\ $\delta_s: L(s+1) \to L(s)$) are obtained from the restriction of the transfer map $\tr^{V_{s+1}}_{V_s}: BV_{(s+1)+} \to BV_{s+}$.  
\end{prop}
	
\subsection{Calculations in $K$- and $j$-theory}
	In this section, we demonstrate the framework established in Proposition \ref{prop:SymFiltSS} to reprove well-known results (e.g., (\ref{thm:K0MsLs}) and (\ref{cor:HZj})) on maps from Eilenberg-Mac~Lane spectra to $K$-theory and related theories.  This will be a prototype for our new calculations for strict units in the next section.  

\subsubsection{$K$-theory of the symmetric power filtration}
	Let $K$ denote complex $K$-theory, and let $R(-)^\wedge$ denote the representation ring completed at the augmentation ideal.  The Atiyah-Segal completion theorem computes the $K$-theory of classifying spaces of finite groups: \[K^0(BG_+) \cong R(G)^\wedge.\]  Specializing to the case $G = V_s$, we find that \[K^0(BV_{s+}) \cong \mathbb{Z}_p[\hat{V}_s],\] where $\hat{V}_s \cong V_s$ is the group of $\mathbb{C}$-valued characters of $V_s$.  Concretely, fix a primitive $p$-th root of unity $\zeta$.  Given $\vec{v} \in V_s$, we obtain a character \[\chi_{\vec{v}}: \vec{w} \mapsto \zeta^{\vec{v} \cdot \vec{w}}.\]  The group $GL(V_s)$ acts by precomposition on $R(V_s)$, and this extends to an action of the group ring.  Specifically, we have \[(\chi_{\vec{v}} \cdot e_s)(\vec{w}) = \frac{1}{[GL_s:U_s]} \sum_{b, \sigma} (-1)^\sigma \chi_{\vec{v}}(b\sigma \cdot \vec{w}) = \frac{1}{[GL_s:U_s]} \sum_{b, \sigma} (-1)^\sigma \chi_{\sigma^\intercal b^\intercal \vec{v}}(\vec{w}),\] so that 

\begin{prop}
	$\chi_{\vec{v}} \cdot e_s = \frac{1}{[GL_s:U_s]} \sum_{b, \sigma} (-1)^\sigma \chi_{\sigma^\intercal b^\intercal \vec{v}}$.
\end{prop}
	
	We want to compute $K^0 M(s) \cong R(V_s)^\wedge \cdot e_s \cong \mathbb{Z}_p[\hat{V}_s] \otimes_{\mathbb{Z}_p[GL(V_s)]} \mathrm{St}'$, where $\mathrm{St}'$ is a Steinberg summand.  
\begin{thm}
	\label{thm:K0Ms}
	The $K$-theory of the Steinberg summands $M(s)$ are:
\begin{equation*}
	K^0 M(s) \cong \begin{cases}
		\mathbb{Z}_p, & s = 0 \\
		\mathbb{Z}_p^{\oplus 2}, & s = 1 \\
		\mathbb{Z}_p, & s = 2 \\
		0, & s \geq 3.
	\end{cases}
\end{equation*}
\end{thm}

\begin{cor}
	The $K$-theory of the Steinberg summands $L(s)$ are:
\begin{equation*}
	K^0 L(s) \cong \begin{cases}
		\mathbb{Z}_p, & s = 0, 1 \\
		0, & s \geq 2.
	\end{cases}
\end{equation*}
\end{cor}
\begin{proof}
	Just use $M(s) \simeq L(s) \vee L(s-1)$ and the previous theorem.  
\end{proof}

	We now turn to the proof of (\ref{thm:K0Ms}).  We shall be a little verbose about this, because we would like to have names for the generators of $K^0 M(s)$.  
	
	First, we can easily dispatch with the case $s = 0$, since $M(0) \simeq S^0$, so $K^0 M(0) \cong \mathbb{Z}_p$.  

\begin{prop}[$s = 1$]
	$K^0 M(1) \cong \mathbb{Z}_p \gen{\chi_0} \oplus \mathbb{Z}_p \gen{\sum_{b \in \mathbb{F}_p^\times} \chi_b}$.  
\end{prop}
\begin{proof}
	We have $K^0 M(1) \cong \mathbb{Z}_p[\hat{V}_1] \cdot e_1$.  The group of characters is \[\hat{V}_1 = \{\chi_{\vec{v}} : \vec{w} \mapsto \zeta^{\vec{v} \cdot \vec{w}} \mid \vec{v} \in \mathbb{F}_p\}.\]  If $b\in \mathbb{F}_p^\times$, then $\chi_{\vec{v}} \cdot b = \chi_{b \vec{v}}$.  Consequently, since $e_1 = \frac{1}{p-1} \sum_{b \in \mathbb{F}_p^\times} b$, we have \[\chi_{\vec{v}} \cdot e_1 = \begin{cases} \chi_{\vec{0}}, & \text{if } \vec{v} = \vec{0} \\ \frac{1}{p-1} \sum_{\vec{v} \neq \vec{0}} \chi_{\vec{v}}, & \text{if } \vec{v} \neq \vec{0}. \end{cases}\]
\end{proof}

\begin{rem}
	This also follows from $M(1) \simeq B\Sigma_{p+}$ and the Atiyah-Segal theorem.
\end{rem}

	We'll need a little more organization to handle the case $s \geq 2$.  Define an equivalence relation on $V_s$ (and hence $\hat{V}_s$) such that $\vec{v} \sim \vec{v}'$ iff we have an equality of multisets $\{(-1)^\sigma \sigma^\intercal b^\intercal \vec{v}\}$ and $\{(-1)^\sigma \sigma^\intercal b^\intercal \vec{v}'\}$ as $b$ and $\sigma$ range over $B_s$ and $\Sigma_s$ respectively.  
	
\begin{lem}
	The equivalence relation $\sim$ partitions $\hat{V}_s$ into exactly $s + 1$ equivalence classes.
\end{lem}
\begin{proof}
	Just by considering the action of $b \in B_s$, we see that every vector $\vec{v} = (v_1, v_2, \ldots, v_s)$ with $v_1 \neq 0$ is equivalent to $\vec{e}_1 = (1,0,\ldots,0)$.  Similarly, every vector $\vec{v} = (0, v_2, \ldots, v_n)$ with $v_2 \neq 0$ is equivalent to $\vec{e}_2 = (0,1,\ldots,0)$, and so on.  So every element in $\hat{V}_s$ is equivalent to one of $\vec{e}_1$, $\vec{e}_2$, \ldots, $\vec{e}_s$, or $\vec{0}$.  

	To see that these $s + 1$ equivalence classes are distinct, note that every element in the multiset $\{\sigma^\intercal b^\intercal \vec{e}_i\}$ has at least $i - 1$ zero entries, with some having exactly $i - 1$ zero entries.  

\end{proof}

	Because of the lemma, it suffices to compute $\chi_{\vec{v}} \cdot e_s$ since for $\vec{v}$ running through $\vec{e}_1, \vec{e}_2, \ldots, \vec{e}_s, \vec{0}$.  
	
\begin{rem}
	In less coordinate-dependent terms, the Borel subgroup $B_s$, or rather $B_s^\intercal$, determines a complete flag on $V_s$, and the action of $B_s^\intercal$ on $V_s$ is transitive on each stratum.  So we just need to compute $\chi_{\vec{v}} \cdot e_s$ for a choice of representative $\vec{v}$ in each stratum.  
\end{rem}
		
	Note firstly that if $s \geq 2$, then the signs $(-1)^\sigma$ cancel out in $\chi_{\vec{0}} \cdot e_s$ to give $0$.  
	
\begin{prop}[$s = 2$]
	$K^0 M(2) \cong \mathbb{Z}_p \gen{\sum_{b \in \mathbb{F}_p^\times} \chi_{b \vec{e}_1} - \chi_{b \vec{e}_2}}$.
\end{prop}
	
\begin{proof}
	We shall show
	\begin{align*}
		\chi_{\vec{e}_1} \cdot e_2 &= \frac{1}{(p+1)(p-1)} \sum_{b \in \mathbb{F}_p^\times} \chi_{b \vec{e}_1} - \chi_{b \vec{e}_2} \\
		\chi_{\vec{e}_2} \cdot e_2 &= \frac{p}{(p+1)(p-1)} \sum_{b \in \mathbb{F}_p^\times} \chi_{b \vec{e}_2} - \chi_{b \vec{e}_1}.
	\end{align*}

	Let $\mathcal{F}_0 \leq \mathcal{F}_1 \leq \mathcal{F}_2$ be the complete flag associated to the transposed Borel $B' = B_2^\intercal$.  To see the first formula, observe that $\{b \vec{e}_1\}_{b \in B'}$ consists of all the vectors in the open stratum $\mathcal{F}_2 \setminus \mathcal{F}_1$, each with multiplicity \[\abs{\Stab_{B'}(\vec{e}_1)} = \frac{\abs{B'}}{\abs{\mathcal{F}_2 \setminus \mathcal{F}_1}} = \frac{p(p-1)^2}{p(p-1)} = p-1.\]  Then, for each vector $\vec{w} \in \mathcal{F}_2 \setminus \mathcal{F}_1$ we consider $\bar{\Sigma}_2 \cdot \vec{w}$.  If $\vec{w}$ contains two repeated entries, then the signed sum $\bar{\Sigma}_2 \cdot \vec{w}$ is zero.  If $\vec{w}$ contains two nonzero coordinates, then swapping them gives another vector $\vec{w}'$ also in the stratum $\mathcal{F}_2 \setminus \mathcal{F}_1$, and $\bar{\Sigma}_2 \cdot \vec{w} + \bar{\Sigma}_2 \cdot \vec{w}' = 0$, and they cancel in pairs.  So we are just left with the vectors $\vec{w} = b \vec{e}_1$, $b \neq 0$, and $\bar{\Sigma}_2 \cdot \chi_{b \vec{e}_1} = \chi_{b \vec{e}_1} - \chi_{b \vec{e}_2}$.  Each of these expressions occur with multiplicity $p - 1$, which leads to the first formula.  

	The argument for $\chi_{\vec{e}_2} \cdot e_2$ is similar.  The multiset $\{b \vec{e}_2\}_{b \in B'}$ consists of just those vectors in the stratum $\mathcal{F}_1 \setminus \mathcal{F}_0$, each with multiplicity \[\abs{\Stab_{B'}(\vec{e}_2)} = \frac{\abs{B'}}{\abs{\mathcal{F}_1 \setminus \mathcal{F}_0}} = \frac{p(p-1)^2}{p-1} = p(p-1).\]  Applying the signed sum operator $\bar{\Sigma}_2$ gives the second formula.  
	
	Finally, note that $\chi_{\vec{e}_2} \cdot e_2 = -p(\chi_{\vec{e}_1} \cdot e_2)$, so the image of the Steinberg idempotent $e_2$ is $1$-dimensional.  
\end{proof}

	Lastly, we have
\begin{prop}
	\label{prop:K1Acyclic}
	If $s \geq 3$, then $\chi_{\vec{v}} \cdot e_s = 0$, so $K^0 M(s) = 0$.  
\end{prop}
\begin{proof}
	We shall show that all the vectors occuring in $\{\sigma^\intercal b^\intercal \vec{v}\}$ cancel.  First, note that as above if $\vec{w} = b^\intercal \vec{v}$ has a repeated coordinate, then there is an odd transposition that fixes $\vec{w}$, so $\bar{\Sigma}_s \cdot \vec{w} = 0$.  So may assume that $\vec{w}$ has no repeated entries.  Since $s \geq 3$, this means that there are at least two nonzero entries.  Let $\vec{w}'$ be the permutation of $\vec{w}$ obtained by swapping the first two nonzero entries.  Note that $\vec{w}'$ belongs to the same stratum (as determined by the flag associated to the Borel) as $\vec{w}$, i.e., there exists $b$ such that $\vec{w}' = b^\intercal \vec{w}$.  Then we have $\bar{\Sigma}_s \cdot \vec{w} + \bar{\Sigma}_s \cdot \vec{w}' = 0$.  So the vectors $\vec{w} \in \{b^\intercal \vec{v}\}$ with no repeated coordinates cancel in pairs.  
\end{proof}

\begin{rem}
	This also follows from the fact that the later Steinberg summands are $K(n)$-acyclic.  In particular, if we're only interested in $K(1)$-local information, we needn't look very deep into the symmetric product filtration.  See Welcher \cite{Wel81} for more details.  
\end{rem}

	This completes the proof of (\ref{thm:K0Ms}).  Let's abbreviate the generators of $K^0 M(s)$ as follows:
	\begin{itemize}
		\item $K^0 M(0) \cong \mathbb{Z}_p$ is generated by $1$.
		\item $K^0 M(1) \cong \mathbb{Z}_p^{\oplus 2}$ is generated by $\alpha = \chi_0$ and $\beta = \sum_{b \in \mathbb{F}_p^\times} \chi_b$.
		\item $K^0 M(2) \cong \mathbb{Z}_p$ is generated by $\gamma = \sum_{b \in \mathbb{F}_p^\times} \chi_{(b,0)} - \chi_{(0,b)}$.  
	\end{itemize}

	The next step is to study the transfer homomorphisms $\delta_s$ in $K$-theory.  From our description of the transfer maps as a stable map between classifying spaces, it may not be immediately obvious what the transfers do.  However, Kahn-Priddy \cite{KP72} has shown that the stable transfer coincides with Atiyah's definition of the transfer in $K$-theory, which in algebra corresponds to induction of representations.  In fact, our situation simplifies further, and the map induced by the transfer is just given by \[M \mapsto \Ind_{V_s}^{V_{s+1}}(M) \cong \rho \boxtimes M,\] where $\rho$ is the regular representation of the quotient cyclic group $C_p := V_{s+1} / V_s$.  

\begin{thm}
\label{thm:K0MsLs}
	The complex \[0 \to K^0 M(0) \to K^0 M(1) \to K^0 M(2) \to 0\] is exact, and the complex \[0 \to K^0 L(0) \to K^0 L(1) \to 0\] is also exact.
\end{thm}
\begin{proof}
	The first transfer map $K^0 M(0) \to K^0 M(1)$ sends $1$ to \[\rho \cdot e_1 = \sum_{a \in \mathbb{F}_p} \chi_a \cdot e_1 = \chi_0 + (p-1) \cdot \frac{1}{p-1} \sum_{b \in \mathbb{F}_p^\times} \chi_b = \alpha + \beta.\]  The second transfer map $K^0 M(1) \to K^0 M(2)$ sends
	\begin{align*}
		\alpha = \chi_0 &\mapsto \sum_{a \in \mathbb{F}_p} \chi_{(a,0)} \cdot e_2 = (p-1) \chi_{\vec{e}_1} \cdot e_2 = (p-1) \cdot \frac{1}{(p+1)(p-1)} \gamma = \frac{1}{p+1} \gamma \\
		\beta = \sum_{b \in \mathbb{F}_p^\times} \chi_b &\mapsto \sum_{\substack{a \in \mathbb{F}_p \\ b \in \mathbb{F}_p^\times}} \chi_{(a,b)} \cdot e_2 = (p-1)^2 \chi_{\vec{e}_1} \cdot e_2 + (p-1) \chi_{\vec{e}_2} \cdot e_2 \\ &\qquad = (p-1)^2 \cdot \frac{1}{(p+1)(p-1)} \gamma + (p-1) \cdot \frac{-p}{(p+1)(p-1)} \gamma = \frac{-1}{p+1} \gamma.
	\end{align*}
	
	Thus our complex becomes \[0 \to \mathbb{Z}_p \{1\} \xrightarrow{\begin{pmatrix} 1 \\ 1 \end{pmatrix}} \mathbb{Z}_p \{\alpha, \beta\} \xrightarrow{\begin{pmatrix} \frac{1}{p+1} & \frac{-1}{p+1} \end{pmatrix}} \mathbb{Z}_p \{\gamma\} \to 0,\] which is exact as claimed. 
	
	Similarly, the complex $0 \to K^0 L(0) \to K^0 L(1) \to 0$ is just \[0 \to \mathbb{Z}_p \{1\} \xrightarrow{\cong} \mathbb{Z}_p \{\alpha + \beta\} \to 0,\] which is also exact.  
\end{proof}

\begin{rem}
	The calculation here is sensitive to the inclusion $V_1 \hookrightarrow V_2$ in the choice of transfer, which comes down to a compatibility with the choice of Borel used to defined the Steinberg summand.  Explicitly, our choice here inserts $V_s$ as the \emph{last} $s$ coordinates in $V_{s+1}$.  
\end{rem}

	By keeping track of the Bott isomorphism, we can do the same calculation for $K^*$ in place of $K^0$ and obtain the same results.  This means that the $E_2$-page of the spectral sequence in (\ref{prop:SymFiltSS}) for $K$-theory is zero and the spectral sequence collapses.  Of course, this is of no surprise since we already know $\Map(H\mathbb{Z}, KU)$ and $\Map(H\mathbb{Z}/p, KU)$ are contractible.  
	
\begin{rem}
	Warning: \emph{before} $p$-completion, the \emph{spectra} $\underline{\Map}(H\mathbb{Z}, KU)$ and $\underline{\Map}(H\mathbb{Z}/p, KU)$ are a lot more interesting.  For example, $KU^1(H\mathbb{Z}) \cong \Ext^1(\mathbb{Q}, \mathbb{Z})$ which is a $\mathbb{Q}$-vector space of uncountable dimension.  
\end{rem}
	
\subsubsection{$j$-theory of the symmetric power filtration}	
	An elaboration of the argument for $K$-theory also works for $j$-theory.  In some sense, this gets us closer to the goal of computing the strict units of the sphere spectrum.  Let us first explain why.  
	
	Assume for now that $p$ is an \emph{odd} prime, and let $l \in \mathbb{Z}_p^\times$ be a topological generator.  We define the spectrum $j$ by the fiber sequence \[j \to bu \xrightarrow{\psi^l - 1} bu,\] where $\psi^l$ is the $l$-th Adams operation.  The complex $J$-homomorphism furnishes a map $\Sigma^{-1} bu \to sl_1 \mathbb{S}$, and the stable Adams conjecture, valid at odd primes, asserts that the $J$-homomorphism factors through a map $j \to sl_1 \mathbb{S}$.  Moreover, the composite \[j \to sl_1 \mathbb{S} \to L_{K(1)} sl_1 S^0 \simeq L_{K(1)} S^0 \langle 1 \rangle\] is an equivalence, which exhibits a $K(1)$-local splitting of the spectrum $j$ from $sl_1 S^0$.  See \cite[VIII.4]{MQRT}.  

\begin{rem}
	The remaining summand is the cokernel of $J$.  One might wonder whether we could continue using the chromatic filtration together with the symmetric product filtration to analyze maps from $H\mathbb{Z}$ into the part of the spectrum of units not part of $j$, but this seems infeasible at this point since -- as we will see -- embedded in this height 1 calculation is a determination of the homotopy groups of the $K(1)$-local sphere.  Thus we expect higher height analogues of this calculation to be at least as difficult as analyzing the $K(n)$-local sphere.  
\end{rem}

	The calculations below are very similar to the ones for $K$-theory, so we'll just sketch some of the modifications needed.
	
	Let $J = \Omega^\infty j$.  By adjunction, $[\Sigma^\infty BV_{s+}, j] \cong [BV_{s+}, J]$.  
\begin{prop}
\mbox{}
\begin{enumerate}[(a)]
	\item	$[BV_{s+}, J]$ is free of rank $\frac{p^s - 1}{p-1}$.
	\item $[BV_{s+}, \Omega^{2i} J] = 0$ for all $i > 0$.  
	\item $[BV_{s+}, \Omega^{2i-1} J] = 0$ if $i$ is not a multiple of $p-1$.
	\item If $i > 0$ is a multiple of $p-1$, then $[BV_{s+}^, \Omega^{2i-1} J] \cong \mathbb{Z}/p^{\nu_p(i) + 1}$, where $\nu_p(i)$ is the exponent of the largest power of $p$ dividing $i$.  Moreover, we may take the trivial character $\chi_{\vec{0}}$ as a generator.  
\end{enumerate}
\end{prop}
\begin{proof}
	The space $J$ sits in a fiber sequence \[J \to BU \xrightarrow{\psi^l - 1} BU,\] where $l \in \mathbb{Z}_p^\times$ is the generator chosen above.  There is a long exact sequence \[\cdots \to [BV_{s+}, \Omega BU] \to [BV_{s+}, J] \to [BV_{s+}, BU] \xrightarrow{\psi^l - 1} [BV_{s+}, BU].\]  Again, the Atiyah-Segal completion theorem tells us that $[BV_{s+}, BU] \cong R_0(V_s)^\wedge$ is the completion of the ring of complex representations of $V_s$ with virtual rank $0$, and $[BV_{s+}, \Omega BU] = 0$.  Therefore,\[[BV_{s+}, J] \cong \ker(\psi^l - 1 \curvearrowright R_0(V_s)^\wedge).\]  
	
	An element of $R_0(V_s)^\wedge$ looks like $\sum_{\vec{v} \in \hat{V}_s} c_{\vec{v}} \chi_{\vec{v}}$ where the coefficients $c_{\vec{v}} \in \mathbb{Z}_p$ sum to zero and $\chi_{v}: \vec{w} \mapsto \zeta^{\vec{v} \cdot \vec{w}}$ is a character of $V_s$.  Given such an element, we have \[(\psi^l - 1)(\sum_{\vec{v}} c_{\vec{v}} \chi_{\vec{v}}) = \sum_{\vec{v}} c_{\vec{v}} \chi_{l\vec{v}} - c_{\vec{v}} \chi_{\vec{v}} = \sum_{\vec{v}} (c_{l^{-1} \vec{v}} - c_{\vec{v}}) \chi_{\vec{v}}.\]  This is zero iff $c_{l^{-1} \vec{v}} = c_{\vec{v}}$ for all $\vec{v}$.  
	
	Recall that $l$ reduces to a generator in $\mathbb{F}_p^\times$, so the equivalence relation $l^{-1} \vec{v} \sim \vec{v}$ partitions $\hat{V}_s$ into $\{0\} \cup \mathbb{P}(\hat{V}_s)$.  If $\vec{v} \neq \vec{0}$, write $c_{\vec{v}} = c_\ell$ where $\ell$ is the line spanned by $\vec{v}$.  The rank zero condition translates to \[c_{\vec{0}} + \sum_{\ell \in \mathbb{P}(\hat{V}_s)} (p-1) c_\ell = 0.\]  Hence, $[BV_{s+}, J]$ consists of the elements \[\left\{ \sum_{\ell \in \mathbb{P}(\hat{V}_s)} c_\ell \sum_{\vec{0} \neq \vec{v} \in \ell} (\chi_{\vec{v}} - \chi_{\vec{0}})\right\}.\]  In particular, it is free of rank $\abs{\mathbb{P}(\hat{V}_s)} = \frac{p^s - 1}{p-1}$.  This proves (a).  
	
	For higher homotopy groups, we have exact sequences \[0 \to [BV_{s+}, \Omega^{2i} J] \to [BV_{s+}, \Omega^{2i} BU] \xrightarrow{\psi^l - 1} [BV_{s+}, \Omega^{2i} BU] \to [BV_{s+}, \Omega^{2i-1} J] \to 0,\] and we want to compute the kernel and cokernel of $\psi^l - 1$ acting on $\Omega^{2i} BU$.  By examining the action of the Adams operation $\psi^l$ on the Bott element, we see that the action of $\psi^l$ is given by \[[BV_{s+}, \Omega^{2i} BU] \cong [BV_{s+}, \mathbb{Z} \times BU] \xrightarrow{\psi^l / l^i} [BV_{s+}, \mathbb{Z} \times BU] \cong [BV_{s+}, \Omega^{2i} BU].\]  
	
	Repeating the calculation from above, we now find \[(\psi^l - 1)(\sum_{\vec{v}} c_{\vec{v}} \chi_{\vec{v}}) = \sum_{\vec{v}} (\frac{c_{l^{-1} \vec{v}}}{l^i} - c_{\vec{v}}) \chi_{\vec{v}}.\]  We can think of this transformation rule on $\{c_{\vec{v}}\}$ as the matrix given by the block sum of the $(1 \times 1)$-block $(\frac{1}{l^i} - 1)$ (corresponding to $\chi_{\vec{0}}$) and $\frac{p^s - 1}{p-1}$ copies of the $((p-1) \times (p-1))$-block \[\begin{pmatrix} -1 & \frac{1}{l^i} \\ & -1 & \frac{1}{l^i} \\ & & \ddots & \ddots  \\ & & & \ddots & \frac{1}{l^i} \\ \frac{1}{l^i} & & & & -1 \end{pmatrix}.\]  The image of this matrix depends on $l$ and $i$.  In particular, the determinant of this matrix is \[(\frac{1}{l^i} - 1)((-1)^{p-1} + \frac{1}{l^{i(p-1)}})^{\frac{p^s - 1}{p-1}}.\]  
	
	We claim that the factor $(-1)^{p-1} + \frac{1}{l^{i(p-1)}}$ is a $p$-adic unit. Since $p$ is odd, it suffices to show that $1 + l^{i(p-1)} \in \mathbb{Z}_p^\times$.  Modulo $p$, this is equivalent to $2$, which is invertible again using $p$ is odd.  
	
	The determination of the $p$-adic valuation of the remaining factor $\frac{1}{l^i} - 1$, or equivalently $1 - l^i$, is exactly the same as in the computation of the homotopy groups of the $K(1)$-local sphere.  Since $1 - l^i$ is not a zero divisor, the map $\psi^l - 1$ is always injective.  This proves (b).  
	
	If $i$ is not a multiple of $p-1$, then $1 - l^i$ is a unit modulo $p$.  Otherwise, write $i = (p-1)p^k i'$ where $i'$ is coprime to $p$.  The element $l^{p-1}$ is a generator for $(1 + p\mathbb{Z}_p)^\times$, and so $l^{(p-1)p^k i'}$ generates $(1 + p^{k+1} \mathbb{Z}_p)^\times$, and $1 - l^{(p-1)p^k i'}$ generates $p^{k+1} \mathbb{Z}_p$.  This proves the remaining parts of the proposition.  
\end{proof}

\begin{prop}
	For $s > 0$,
\begin{enumerate}[(a)]
	\item $[M(s), j] \cong \mathbb{Z}_p$ if $s = 1, 2$, and is zero if $s \geq 3$.  
	\item $[M(s), \Omega^i j] \cong \mathbb{Z}/p^{k+1}$ if $s = 1$ and $i = 2(p-1)p^k i' - 1$ with $i'$ prime to $p$, and is zero otherwise.  
\end{enumerate}
\end{prop}
\begin{proof}
	For part (a), when $s = 1$, the Steinberg idempotent $e_1$ acts on $[BV_{1+}, J] \cong \mathbb{Z}_p \gen{\sum_{\vec{v} \neq \vec{0}} (\chi_{\vec{v}} - \chi_{\vec{0}})}$ by the identity, so $[M(1), j] \cong \mathbb{Z}_p$.  
	
	When $s = 2$, $[BV_{2+}, J] \cong \mathbb{Z}_p \gen{\sum_{\vec{0} \neq \vec{v} \in \ell} \chi_{\vec{v}} - \chi_{\vec{0}} \mid \ell \in \mathbb{P}(\hat{V}_2)}$, and we have \[\left(\sum_{\vec{0} \neq \vec{v} \in \ell} \chi_{\vec{v}} - \chi_{\vec{0}} \right) \cdot e_2 = \begin{cases} \frac{p}{p+1} \sum_{b \in \mathbb{F}_p^\times} \chi_{b \vec{e}_2} - \chi_{b \vec{e}_1}, & \ell = \mathcal{F}_1 \\ \frac{1}{p+1} \sum_{b \in \mathbb{F}_p^\times} \chi_{b \vec{e}_1} - \chi_{b \vec{e}_2}, & \ell \neq \mathcal{F}_1. \end{cases}\]  (Recall that $\mathcal{F}_1$ is the line preserved by the transposed Borel.)  Thus $[M(2), j]$ has rank 1 with generator $\gamma$.  
	
	For $s \geq 3$, the same proof of (\ref{prop:K1Acyclic}) goes through without change.  
	
	For part (b), the group $[BV_{s+}, \Omega^i J]$ is zero for many values of $i$.  In the remaining cases, it is generated by $\chi_{\vec{0}}$.  If $s = 1$, the Steinberg idempotent $e_1$ acts by the identity on $\chi_{\vec{0}}$.  If $s > 1$, the Steinberg idempotent $e_s$ annihilates $\chi_{\vec{0}}$.  
\end{proof}	

\begin{prop}
\mbox{}
\begin{enumerate}[(a)]
	\item The homology of the complex \[\cdots \to [M(s), j] \to [M(s+1), j] \to \cdots\] has homology $\mathbb{Z}/p$ in $s$-degree $2$.   
	\item For $i > 0$, the complexes \[\cdots \to [M(s), \Omega^i j] \to [M(s+1), \Omega^i j] \to \cdots\] are exact.  
\end{enumerate}
\end{prop}
\begin{proof}
	For $i = 0$, we have \[\underbrace{[M(0), j]}_{0} \to \underbrace{[M(1), j]}_{\mathbb{Z}_p} \to \underbrace{M(2), j]}_{\mathbb{Z}_p} \to 0 \to \cdots.\]  Using the notation introduced in the last section, the map induced by the transfer sends the generator $\beta - (p-1)\alpha$ of $[M(1), j]$ to \[\frac{-1}{p+1} \gamma - (p-1) \cdot \frac{1}{p+1} \gamma = \frac{-p}{p+1} \gamma.\]  So the transfer map $[M(1), j] \to [M(2), j]$ has zero kernel and cokernel $\mathbb{Z}/p$.  
	
	For $i = 2(p-1)p^k i' - 1$, we want to show that \[\tr: \underbrace{[M(0), \Omega^i j]}_{\mathbb{Z}/p^{k+1}} \to \underbrace{[M(1), \Omega^i j]}_{\mathbb{Z}/p^{k+1}}\] is an isomorphism.  We have \[\tr(1) = (\sum_{a \in \mathbb{F}_p} \chi_a) \cdot e_1 \equiv \chi_{\vec{0}} \pmod{\sum_{b \in \mathbb{F}_p^\times} \chi_b},\] as desired.  
\end{proof}

\begin{cor}
\label{cor:MapLsJ}
\mbox{}
\begin{enumerate}[(a)]
	\item $[L(s), j] = 0$ unless $s = 1$, in which case $[L(1), j] \cong \mathbb{Z}_p$.  
	\item For $i > 0$, the complexes \[\cdots \to [L(s), \Omega^i j] \to [L(s+1), \Omega^i j] \to \cdots\] are exact.
\end{enumerate}
\end{cor}

\begin{cor}
\label{cor:HZj}
	$\Map(H\mathbb{Z}, j)$ is contractible.  
\end{cor}
\begin{proof}
	Combine the previous corollary with (\ref{prop:SymFiltSS}).
\end{proof}

	Even though the mapping space $\Map(H\mathbb{Z}, j)$ is contractible, it seems that the mapping \emph{spectrum} has a nonzero homotopy group in $\pi_{-1} \cong \mathbb{Z}_p$.  
	
\begin{prop}
	$\pi_{-1} \underline{\Map}(H\mathbb{Z}, j)$ is contains the non-null map $\Sigma^{-1} H\mathbb{Z} \to L_{K(1)} S^0 \gen{1} \simeq j$ coming from the fiber sequence \[\Sigma^{-1} H\mathbb{Z} \to L_{K(1)} S^0 \gen{1} \to L_{K(1)} S^0 \xrightarrow{\pi_0} H\mathbb{Z}.\] 
\end{prop}
\begin{proof}
	First, the proposed map is nonzero.  Otherwise, we would have a splitting $L_{K(1)} S^0 \simeq H\mathbb{Z} \vee L_{K(1)} S^0 \gen{1}$, but there are no nontrivial maps $H\mathbb{Z} \to L_{K(1)} S^0$.  Second, if the map were divisible by $p$, then the composite $\Sigma^{-2} H\mathbb{Z}/p \to \Sigma^{-1} H\mathbb{Z} \to L_{K(1)} S^0 \gen{1}$ is null, so there is a nonzero map $\Sigma^{-2} H\mathbb{Z}_p \to L_{K(1)} S^0$, which is impossible.  
\end{proof}


\subsection{Calculations in the units of the Burnside ring}
	Starting in this section, we finally tackle the units of the sphere spectrum.  Recall from (\ref{prop:SymFiltSS}) that there is a spectral sequence computing $\pi_* \mathbb{G}_m(S^0)$ with $E_1^{s,t} = [\Sigma^t L(s), gl_1 S^0]$.  For convenience, we now assume that $p$ is an odd prime, and remind the reader that everything is implicitly $p$-completed.  The main theorem of this section is:
\begin{thm}
\label{thm:LsGL1S0}
	For this spectral sequence, $E_2^{s,0} = 0$ for all $s$.  
\end{thm}

	Compare this to (\ref{cor:MapLsJ}). 

	Again, we shall first prove the corresponding result for the complex involving $[M(s), gl_1 S^0]$.  
\begin{thm}
\label{thm:MsGL1S0}
	The spectral sequence converging to $\pi_* \Map(H\mathbb{Z}/p, gl_1 S^0)$ with $E_1^{s,t} = [\Sigma^t M(s), gl_1 S^0]$ has \[E_2^{s,0} \cong \begin{cases} \mathbb{Z}/p, & s = 2 \\ 0, & \text{otherwise}. \end{cases}\]
\end{thm}
	
	Like $K$-theory, the sphere $S^0$ satisfies a kind of ``completion theorem'': the Segal conjecture.  The calculations needed here are essentially the same as in the previous section, except now the representation ring is replaced by the Burnside ring as per the Segal conjecture, and induction of representations is replaced by the multiplicative \emph{norm} in the Burnside ring.  
	
	Recall, the Burnside ring of a finite group $G$ is the Grothendieck group associated to the monoid of finite $G$-sets under disjoint union; it is equipped with a multiplication coming from the Cartesian product of $G$-sets.  More concretely, since every $G$-set is a disjoint union of transitive $G$-sets, an element in the Burnside ring can be written as a $\mathbb{Z}$-linear combination of the orbits $[G/H]$ for $H \leq G$.  
	
	As we vary the group $G$, the Burnside rings form a \emph{Tambara functor}.  In particular, given $H \leq G$, we have \emph{transfers} $\tr_H^G: A(H) \to A(G)$ and \emph{norms} $N_H^G: A(H) \to A(G)$.  They are given by the following formulae: for a finite $H$-set $X$, we have
	\begin{align*}
		\tr_H^G(X) &= X \times_H G, \\
		N_H^G(X) &= \Map_H(G,X).
	\end{align*}
	
	The Segal conjecture relates the Burnside ring $A(G)$ with $\pi_0^G$ of the $G$-equivariant sphere, which is a $G$-$E_\infty$-spectrum.  In particular, $\pi^0(BG_+)$ also comes with transfers and norms.  
\begin{thm}[Segal conjecture]
	There is a ring homomorphism $A(G) \to \pi^0(BG_+)$, which restricts to a group homomorphism $A(G)^\times \to (gl_1 S^0)^0(BG_+)$.  In fact, these homomorphisms are isomorphisms after a suitable completion.  
\end{thm}

	Moreover, we have the following intuitive proposition; see \cite[Lem.~1.2]{Min91} for a proof.
\begin{prop}
\label{prop:BurnsideTransfer}
\mbox{}
\begin{enumerate}[(a)]
	\item The transfer in $\pi^0$ coincides with the transfer in the Burnside ring.  
	\item The transfer in $(gl_1 S^0)^0$ coincides with the norm in the Burnside ring.  
\end{enumerate}
\end{prop}

	We first compute the groups $[M(s), gl_1 S^0]$.  
\begin{prop}
\mbox{}
\begin{enumerate}[(a)]
	\item $[M(0), gl_1 S^0] \cong (1 + p\mathbb{Z}_p)^\times$.
	\item $[M(1), gl_1 S^0] \cong (1 + p\mathbb{Z}_p)^\times \times (1 + p\mathbb{Z}_p)^\times$.
	\item $[M(2), gl_1 S^0] \cong (1 + p\mathbb{Z}_p)^\times$.
	\item $[M(s), gl_1 S^0] = 1$ for $s \geq 3$.  
\end{enumerate}
\end{prop}

\begin{proof}
\mbox{}
\begin{enumerate}[(a)]
	\item $M(0) = \mathbb{S}$, so $[M(0), gl_1 S^0] \cong (\mathbb{Z}_p^\times)^\wedge_p \cong (1 + p\mathbb{Z}_p)^\times$.  (Recall that for an odd prime $p$, we have a decomposition $\mathbb{Z}_p^\times \cong \mu_{p-1} \times (1 + p\mathbb{Z}_p)^\times$.)
	
	\item We have \[[M(1), gl_1 S^0] \cong A(V_1)^\times \cdot e_1,\] and $b \in GL(V_1)$ acts identically on $A(V_1)$, so for $u \in A(V_1)^\times$, we have \[u \cdot e_1 = \left( \prod_{b \in GL(V_1)} u \cdot g \right)^{1/(p-1)} = u.\]  Thus the Steinberg idempotent acts by the identity, so \[[M(1), gl_1 S^0] \cong A(V_1)^{\times \wedge} \cong (\mathbb{Z}_p[x]/(x^2 - px))^{\times \wedge} \cong ((1 + p\mathbb{Z}_p)^\times)^2.\]  
	
	\item The multiplication in the Burnside ring is somewhat complicated, and it is convenient to work in its ghost ring.  Recall that for a finite group $G$, there is an injective mark homomorphism \[\phi: A(G) \hookrightarrow \bar{A}(G) := \mathbb{Z}^{\Cl(G)}\] that sends a $G$-set $X$ to the function that takes a subgroup $C$ up to conjugacy to the number $\abs{X^C}$ of fixed points.  
	
	When $G$ is an $\mathbb{F}_p$-vector space $V$, there is a right action of $GL(V)$ on $A(V)$ given by $X \cdot g = g^{-1} X$.  We transport this action through the mark homomorphism.  Define an action of $GL(V)$ on $\Cl(V)$ by $g \cdot W = gWg^{-1}$ where $W \leq V$ is a subspace and the conjugation takes place inside the larger group $\operatorname{Aff} := V \rtimes GL(V)$.  Endow the ghost ring $\bar{A}(V) = \mathbb{Z}^{\Cl(V)}$ with the induced action.  Then,
	
	\begin{claim}
		The mark homomorphism $\phi: A(V) \to \bar{A}(V)$ is $GL(V)$-equivariant, i.e., \[\phi(X \cdot g)(W) = \phi(X)(g \cdot W).\]
	\end{claim}
	
	To see this claim, note that $\phi(X \cdot g)(W) = \abs{(X \cdot g)^W}$ and $\phi(X)(g \cdot W) = \abs{X^{gWg^{-1}}}$, so we want to exhibit a bijection between $(X \cdot g)^W$ and $X^{gWg^{-1}}$ for all $X \in A(V)$, $g \in GL(V)$, and $W \leq V$.  Given $x \in X \cdot g = g^{-1} X$, we have $gx \in X$.  If $x$ is $W$-fixed, then \[gWg^{-1} \cdot gx = gW \cdot x = gx,\] so $gx$ is $gWg^{-1}$-fixed.  This proves the claim.  
	
	Let $X \in A(V)^\times$.  Using this claim, we obtain \[\phi(X \cdot e_s)(W) = \phi(X)(e_s \cdot W) = \left(\prod_{\substack{b \in B_s \\ \sigma \in \Sigma_s}} \phi(X)(b\sigma \cdot W)^\sigma \right)^{\frac{1}{[GL_s:U_s]}}.\]  
	
	Now we specialize to the case $s = 2$.  There are $p + 3$ subgroups of $V_2$: the trivial subspace $\{0\}$, the whole subspace $V_2$, and the $p + 1$ lines in $\mathbb{P}(V_2)$.  The first two cases are easy: the action of $GL(V)$ has to act identically on $\{0\}$ and on $V_2$ by dimensional reasons, and we can see that
	\begin{align*}
		\phi(X \cdot e_2)(V_2) &= \left( \prod_{b, \sigma} \phi(X)(V_2)^\sigma \right)^{1/(p+1)(p-1)^2} = 1 \\
		\phi(X \cdot e_2)(\{0\}) &= \left( \prod_{b, \sigma} \phi(X)(\{0\})^\sigma \right)^{1/(p+1)(p-1)^2} = 1, \\
	\end{align*}
	the sign of $\sigma \in \Sigma_2$ cancelling everything in pairs.  
	
	The remaining case of the lines is harder.	One of these lines is distinguished by the Borel and the associated flag $\mathcal{F}_0 \leq \mathcal{F}_1 \leq \mathcal{F}_2$ on $V_2$.  First consider $\ell \in \mathbb{P}(V_2) \setminus \{\mathcal{F}_1, \sigma \mathcal{F}_1\}$, and set $\mathbb{A}(V_2) := \mathbb{P}(V_2) \setminus \{\mathcal{F}_1\}$.  Then as multisets, we have \[\{b \cdot \ell\}_{b \in B_2} = \{b\sigma \cdot \ell\}_{b \in B_2} = \{(p-1)^2 \times L\}_{L \in \mathbb{A}(V_2)},\] and so we can cancel in pairs again and find \[\phi(X \cdot e_2)(\ell) = 1.\]  
	
	Finally, we need to deal with the two remaining lines $\mathcal{F}_1$ and $\sigma \mathcal{F}_1$.  We have
	\begin{align*}
		\{b \cdot \mathcal{F}_1\}_{b \in B_2} &= \{p(p-1)^2 \times \mathcal{F}_1\} \\
		\{b \cdot \sigma \mathcal{F}_1\}_{b \in B_2} &= \{(p-1)^2 \times L\}_{L \in \mathbb{A}(V_2)}.
	\end{align*}
	
	\noindent Consequently,
	\begin{align*}
		\phi(X \cdot e_2)(\mathcal{F}_1) &= \left(\frac{\phi(X)(\mathcal{F}_1)^{p(p-1)^2}}{\prod_{L \in \mathbb{A}(V_2)} \phi(X)(L)^{(p-1)^2}}\right)^{1/(p+1)(p-1)^2} \\
		&= \left(\frac{\phi(X)(\mathcal{F}_1)^p}{\prod_{L \in \mathbb{A}(V_2)} \phi(X)(L)}\right)^{1/(p+1)} \\
		\phi(X \cdot e_2)(\sigma \mathcal{F}_1) &= \left(\frac{\prod_{L \in \mathbb{A}(V_2)} \phi(X)(L)}{\phi(X)(\mathcal{F}_1)^p}\right)^{1/(p+1)}.
	\end{align*}
	
	If we want, we can invert the mark homomorphism.  An additive basis for $A(V_2)$ is given by $1 = [V_2/V_2]$, $x_\ell = [V_2/\ell]$ for $\ell \in \mathbb{P}(V_2)$, and $y = [V_2/\{0\}]$.  Let $X = a + \sum_\ell b_\ell x_\ell + cy$ for some coefficients $a, b_\ell, c$.  Then we can count fixed points and we have
	\begin{equation*}
	\begin{array}{cc}
		\hline
		W & \phi(X)(W) \\ \hline
		V_2 & a \\
		\ell & a + pb_\ell \\
		\{0\} & a + p\sum_\ell b_\ell + p^2 c \\ \hline
	\end{array}
	\end{equation*}
	
	It follows that \[X = \phi(X)(V_2) + \sum_\ell \frac{\phi(X)(\ell) - \phi(X)(V_2)}{p} x_\ell + \frac{\phi(X)(\{0\}) - \phi(X)(V_2) - \sum_\ell (\phi(X)(\ell) - \phi(X)(V_2))}{p^2} y.\]  
	
	Plugging in the values for $X \cdot e_2$ computed above, we get \[X \cdot e_2 = 1 + \frac{t-1}{p} x_{\mathcal{F}_1} + \frac{t^{-1} - 1}{p} x_{\sigma \mathcal{F}_1} + \frac{(t-1)(t^{-1} - 1)}{p^2} y,\] where \[t = t(X) := \left(\frac{\phi(X)(\mathcal{F}_1)^p}{\prod_{L \in \mathbb{A}(V_2)} \phi(X)(L)}\right)^{1/(p+1)} = \left(\frac{(a + pb_{\mathcal{F}_1})^p}{\prod_{L \in \mathbb{A}(V_2)} a + pb_L}\right)^{1/(p+1)}.\]  
	
	Examining the formula for $t$, we see that $t(X)$ ranges over $1 + p\mathbb{Z}_p$ as $X$ varies, and we deduce that $[M(2), gl_1 S^0]$ is a one-parameter family parametrized by $t \in 1 + p\mathbb{Z}_p$.  
	
	\item Finally, suppose $s \geq 3$.  To have complete cancellation of all the factors, we need to show that for each subspace $W \in \Gr_k(V_s)$, the $\Sigma_s$-orbit of $W$ in $\Gr_k(V_s)$ can be partitioned into pairs such that if $\sigma_1 W$ is matched with $\sigma_2 W$ for $\sigma_1, \sigma_2 \in \Sigma_s$, then
	\begin{enumerate}[(i)]
		\item the permutation $\sigma_1 \sigma_2$ is odd, and
		\item $\sigma_1 W$ and $\sigma_2 W$ belong in the same $B_s$-orbit.  
	\end{enumerate}
	
	We construct an involution $\tau$ on the $\Sigma_s$-orbit of $W$ that realize a matching satisfying the above properties, i.e.,
	\begin{enumerate}[(i)]
		\item $W$ and $\tau(W)$ differs by a transposition for any $W \in \Gr_k(V_s)$, and
		\item $W$ and $\tau(W)$ are in the same $B_s$-orbit.  
	\end{enumerate}
	
	Pick an orthonormal basis $\{\vec{e}_i\}$ adapted to the Borel and the symmetric group.  Let $W' \in \Sigma_s \cdot W$.  Suppose $\abs{\{i \mid W' \cap \vec{e}_i^\perp \neq \{0\}\}} \geq 2$.  (Note that this property is invariant under $\Sigma_s$.)  Pick the two smallest elements $i_1$ and $i_2$ in this set and define $\tau(W') = \sigma_{i_1, i_2} W'$.  We check that this defines an involution.  Observe that $\tau(W) \cap \vec{e}_{i_1}^\perp = \sigma_{i_1, i_2}(W \cap \vec{e}_{i_2}^\perp) \neq \{0\}$, and $\tau(W) \cap \vec{e}_{i_2}^\perp = \sigma_{i_1,i_2} (W \cap \vec{e}_{i_1}^\perp) \neq \{0\}$.  Moreover, if $i \neq i_1, i_2$ is such that $\tau(W) \cap \vec{e}_i^\perp \neq \{0\}$, then $W \cap \vec{e}_i^\perp \neq \{0\}$, so $i > \max(i_1,i_2)$.  This shows that the indices $i_1, i_2$ for $W$ and $\tau(W)$ are the same.  
	
	We next check that $W$ and $\tau(W)$ are in the same $B_s$-orbit.  Let $\{\mathcal{F}_0 \leq \mathcal{F}_1 \leq \cdots \leq \mathcal{F}_s\}$ be the complete flag of $V_s$ that is preserved by $B_s$.  The $B_s$-orbits of $\Gr_k(V_s)$ are of the form $\Gr_k(\mathcal{F}_d) \setminus \Gr_k(\mathcal{F}_{d-1})$.  
		
		Suppose $W \in \Gr_k(\mathcal{F}_d) \setminus \Gr_k(\mathcal{F}_{d-1})$.  Since $W \in \Gr_k(\mathcal{F}_d)$, we have $W \cap \vec{e}_i^\perp = \{0\}$ for $i \geq d + 1$.  This implies $i_1, i_2 \leq d$, and that $\tau(x) \cap \vec{e}_i^\perp = \sigma_{i_1,i_2}(W \cap \vec{e}_{\sigma_{i_1,i_2}(i)}^\perp) = \sigma_{i_1,i_2}(W \cap \vec{e}_i^\perp) = \{0\}$ for $i \geq d + 1$.  Therefore, $\tau(W) \in \Gr_k(\mathcal{F}_d)$.  On the other hand, $\tau(W) \notin \Gr_k(\mathcal{F}_{d-1})$; otherwise, the same argument would show that $W = \tau(\tau(W)) \in \Gr_k(\mathcal{F}_{d-1})$, which is a contradiction.  
		
		It remains to define $\tau$ on the 1-dimensional subspaces.  In this case, we can define $\tau$ by ``swapping the first two zero coordinates''.  This requires $s \geq 3$.  
\end{enumerate}
\end{proof}

	Our complex $\{[M(s), gl_1 S^0]\}$ is: \[1 \to (1 + p\mathbb{Z}_p)^\times \to ((1 + p\mathbb{Z}_p)^\times)^2 \to (1 + p\mathbb{Z}_p)^\times \to 1.\]  Because of proposition \ref{prop:BurnsideTransfer}, we can compute these maps in terms of norms in the Burnside ring.  

\begin{prop}
	The homology of this complex is trivial except when $s = 2$, in which case it is equal to $\mathbb{Z}/p$.  
\end{prop}
\begin{proof}
	The first transfer map takes $a \in 1 + p\mathbb{Z}_p$, regarded as a set with $a$ elements, to \[N(a) = \Map(V_1, a)\]  The cardinality of $\Map(V_1, a)$ is $a^p$, and the $V_1$-fixed points of $\Map(V_1, a)$ are the constant functions, of which there are $a$ of them.  In other words, 
	\begin{align*}
		\phi(N(a))(\{0\}) &= a^p \\
		\phi(N(a))(V_1) &= a.
	\end{align*}
	
	Inverting the mark homomorphism, we have \[N(a) = a + \frac{a^p - a}{p} x.\]  In particular, the first map is injective with image consisting of $\{a + \frac{a^p - a}{p} x \mid a \in (1 + p\mathbb{Z}_p)^\times\}$.  
	
	The second transfer map takes $a + bx$ to $N(a + bx)$, where
	\begin{align*}
		\phi(N(a+bx))(\{0\}) &= (a + pb)^p \\
		\phi(N(a+bx))(L) &= \begin{cases} a^p, & L = \mathcal{F}_1 \\ a + pb, & L \neq \mathcal{F}_1 \end{cases} \\
		\phi(N(a+bx))(V_2) &= a.
	\end{align*}
	
	Thus, \[N(a+bx) = a + \frac{a^p - a}{p} x_{\mathcal{F}_1} + b \sum_{L \in \mathbb{A}(V_2)} x_L + \frac{(a+pb)^p - a^p - p^2 b}{p^2} y.\]  To apply the Steinberg idempotent, it suffices to compute \[t(N(a+bx)) = \left(\frac{(a^p)^p}{\prod_{L \in \mathbb{A}(V_2)} a + pb}\right)^{1/(p+1)} = \left(\frac{a^p}{a + pb}\right)^{p/(p+1)}.\]  
	
	We check exactness in the middle.  If $t(N(a+bx)) = 1$, since there are no non-trivial $p$-th roots in $\mathbb{Z}_p$, we have $\frac{a^p}{a+pb} = 1$, or $b = \frac{a^p - a}{p}$.  This precisely determines the image of the first map.  
	
	What about the image?  The image of $(a,b) \mapsto (\frac{a^p}{a+pb})^{1/(p+1)}$ hits all of $1 + p\mathbb{Z}_p$.  Raising this to the $p$-th power, we see that the image of $(a,b) \mapsto t(N(a+bx))$ is $(1 + p\mathbb{Z}_p)^p = 1 + p^2 \mathbb{Z}_p$.  Hence, the cokernel of the second transfer map is $(1 + p^2 \mathbb{Z}_p)/(1 + p\mathbb{Z}_p) \cong \mathbb{Z}/p$.  
\end{proof}

	Here are the analogous results for the $L(s)$'s instead of the $M(s)$'s.  

\begin{prop}
\mbox{}
\begin{enumerate}[(a)]
	\item $[L(0), gl_1 S^0] \cong (1 + p\mathbb{Z}_p)^\times$
	\item $[L(1), gl_1 S^0] \cong (1 + p\mathbb{Z}_p)^\times$
	\item $[L(s), gl_1 S^0] = 1$ for $s \geq 2$.
\end{enumerate}
\end{prop}

\begin{prop}
	The map $[L(0), gl_1 S^0] \to [L(1), gl_1 S^0]$ is an isomorphism.  
\end{prop}
\begin{proof}
	The isomorphism $M(1) \simeq L(0) \vee L(1)$ induces the following decomposition in $(gl_1 S^0)^0$-cohomology:
	\begin{align*}
		(\mathbb{Z}_p[x]/(x^2 - px))^\times &\cong \mathbb{Z}_p^\times \times (1 + p\mathbb{Z}_p)^\times \\
		a + bx &\mapsto (a + pb, 1 + p\frac{b}{a}) \\
		\frac{a}{1+pb}(1 + bx) &\mapsfrom (a, 1 + pb)
	\end{align*}
	
	Thus, the reduced transfer induces the map $a \mapsto a^{p-1}$, which is an isomorphism on $(1 + p\mathbb{Z}_p)^\times$.  
\end{proof}

	When $p = 2$, the calculations differ slightly, for example due to the behaviour of the unit group in the $2$-adic integers: $\mathbb{Z}_2^\times \cong \mathbb{Z}/2 \times (1 + 4\mathbb{Z}_2)^\times$.  The analogues of theorems \ref{thm:LsGL1S0} and \ref{thm:MsGL1S0} are:
\begin{thm}
	Let $p = 2$.  
\begin{enumerate}[(a)]
	\item The spectral sequence converging to $\pi_* \Map(H\mathbb{Z}/2, gl_1 S^0)$ with $E_1^{s,t} = [\Sigma^t L(s), gl_1 S^0]$ has \[E_2^{s,0} \cong \begin{cases} \mathbb{Z}/2 \times \mathbb{Z}/2, & s = 2 \\ 0, & \text{otherwise}. \end{cases}\]
	\item The spectral sequence converging to $\pi_* \Map(H\mathbb{Z}_2, gl_1 S^0)$ with $E_1^{s,t} = [\Sigma^t M(s), gl_1 S^0]$ has $E_2^{s,0} = 0$ for all $s$.
\end{enumerate}
\end{thm}
\begin{proof}
	Same as above, with minor modifications.  
	
\end{proof}

\subsection{Comparison of additive and multiplicative structures}
	In the section, we briefly sketch one possible approach to understanding the spectral sequence (\ref{prop:SymFiltSS}) in the region $t > 0$.  Let $R$ be an arbitrary $E_\infty$-ring again.  

\begin{prop}
	\label{prop:SameGroups}
	For $t > 0$, $[\Sigma^t L(s), gl_1 R] \cong [\Sigma^t L(s), R]$.  
\end{prop}
\begin{proof}
	For convenience we prove the version for the $M(s)$'s in place of the $L(s)$.  This is enough, by stable cancellation \cite[Thm.~10.1.9]{Mar83} in $p$-complete spectra.  
	
	We have $[\Sigma^t BV_{s+}, gl_1 R] \cong [BV_{s+}, \Omega^t GL_1 R]$ and $[\Sigma^t BV_{s+}, R] \cong [BV_{s+}, \Omega^{t + \infty} R]$.  Since $GL_1 R$ is a union of components of $\Omega^\infty R$, their loop spaces $\Omega^t GL_1 R$ and $\Omega^{t + \infty} R$ are homotopy equivalent when $t \geq 1$.  
	
	It remains to show that action of $\mathbb{Z}_p[GL(V_s)]$ on the groups $[BV_{s+}, \Omega^t GL_1 R]$ and $[BV_{s+}, \Omega^{t+\infty} R]$ coincide.  The action of $GL(V_s)$ is clearly the same since it can be defined unstably.  On the other hand, the addition also acts in the same way since the ``addition'' in $\Omega^{t+\infty} R$ and the ``multiplication'' in $\Omega^t GL_1 R$ are both equivalent to the loop concatenation operation by an Eckmann-Hilton argument.  
\end{proof}

\begin{rem}
	For $t \geq 2$, we can offer a slightly quicker alternative proof, using the fact that there exist spaces $L_1(s)$ such that $\Sigma L(s) \simeq \Sigma^\infty L_1(s)$.  (For $s = 0$ and $s = 1$, $L(s)$ themselves are already suspension spectra.)  
	
	We have \[[\Sigma^t L(s), gl_1 R] \cong [\Sigma^{t-2} L_1(s), \Omega GL_1 R] \cong [\Sigma^{t-2} L_1(s), \Omega^{1+\infty} R] \cong [\Sigma^t L(s), R].\]  
\end{rem}

	It is tempting to compare the additive and multiplicative versions of the spectral sequence computing $\Map(H\mathbb{Z},R)$ and $\Map(H\mathbb{Z}, gl_1 R)$, especially in the case $R = S^0$ where it is known that the spectral sequence for $\Map(H\mathbb{Z}, S^0)$ collapses in a particularly nice way \cite{Cat88}. 	However, we caution that even the $d_1$-differentials of the spectral sequences, given by transfers and norms respectively, can behave quite differently.  In a sense, we already know this from the description of the Koszul differential in the homotopy spectral sequence for $E_\infty$-maps in section \ref{sec:GHM}, but we offer a concrete example here.  
	
\begin{exmp}
	Take $R = H\mathbb{F}_2[u]/u^3$ with $\deg u = d > 0$, and let us study the original Kahn-Priddy transfer $\Sigma^\infty \mathbb{R}P^\infty \to S^0$ in $gl_1 R$-cohomology.  Since $(gl_1 H\mathbb{F}_2[u]/u^3)^* S^0 \cong \mathbb{F}_2[u]/u^3$ and $(gl_1 H\mathbb{F}_2[u]/u^3)^* \mathbb{R}P^\infty \cong \mathbb{F}_2[x,u]/u^3$, where $\deg x = -1$, there is only one degree for which the map is possibly nonzero, namely $u \overset{?}{\mapsto} u^2 x^d$ in degree $d$.
	
	Additively, this map has to be zero: it is essentially a formal property that this transfer always induces the zero map in ordinary mod 2 cohomology.  However, we claim that it is nonzero in $(gl_1 R)^*$.  
	
	To see this, we use the fact (\ref{prop:BottomKInvt}) discussed earlier that \[gl_1 H\mathbb{F}_2[u]/u^3 \simeq \fib(\Sq^{d+1}: \Sigma^d H\mathbb{F}_2 \to \Sigma^{2d+1} H\mathbb{F}_2).\]  Consider the composition \[\Sigma^\infty \mathbb{R}P^\infty \xrightarrow{\tr} S^0 \xrightarrow{1} H\mathbb{F}_2 \xrightarrow{\Sq^{d+1}} \Sigma^{d+1} H\mathbb{F}_2.\]  The compositions of the two pairs of adjacent arrows are nullhomotopic, and so this determines a Toda bracket $\langle \Sq^{d+1}, 1, \tr \rangle \subseteq H^d \mathbb{R}P^\infty$.  Moreover, this bracket has zero indeterminacy and represents the transfer of $u$ in $gl_1 R$-cohomology.  
	
	\begin{center}
	\begin{tikzcd}
		& & \Sigma^d H\mathbb{F}_2 \ar[d] \\
		& & \Omega^d gl_1 H\mathbb{F}_2[u]/u^3 \ar[d] \\
		\mathbb{R}P^\infty \ar[r] \ar[rruu, dashed, "{\langle \Sq^{d+1}, 1, \tr \rangle}"] & S^0 \ar[d] \ar[r] \ar[ru, dashed] & H\mathbb{F}_2 \ar[r] & \Sigma^{d+1} H\mathbb{F}_2 \\
		& \Sigma \mathbb{R}P^\infty_{-1} \ar[d] \ar[ur, dashed] \\
		& \Sigma \mathbb{R}P^\infty \ar[uurr, dashed, "{\langle \Sq^{d+1}, 1, \tr \rangle}"']
	\end{tikzcd}
	\end{center}
	
	We can compute the value of this Toda bracket by interpreting it in another way.  Whether it is zero or not depends on whether the bottom cell of $\mathbb{R}P^\infty_{-1}$ is connected by a $\Sq^{d+1}$ to the cell in dimension $d$, and it is a classical fact that this is indeed the case.  Hence we conclude that the transfer takes $u$ to $u^2 x^d$.  
\end{exmp}

	In general, let $x_1, \ldots, x_s$ be a basis of $H^1(BV_s)$ such that the inclusion $H^1(BV_s) \to H^1(BV_{s+1})$ induced by the projection $BV_{s+1} \to BV_s$ is the natural one.  Then,
\begin{prop}
	The transfer $BV_{s+1+} \to BV_{s+}$ sends $ux_1^{i_1} \cdots x_s^{i_s}$ to $u^2 x_1^{i_1} \cdots x_s^{i_s} x_{s+1}^d$.  
\end{prop}

\begin{rem}
	Kozlowski \cite{Koz83} has an even more general formula for this type of transfer in the units of $H\mathbb{F}_2[u]$ in terms of Steenrod operations.  In the case of ordinary cohomology, Steiner \cite{Ste82} has shown that transfers induce the Evens norm in group cohomology.  
\end{rem}
	
	Nonetheless, there is a relationship between the additive and multiplicative transfers.  Let $R$ be an $E_\infty$-ring.  For each $s \geq 0$, we have maps
	\begin{gather*}
		\tr^*_\oplus: [\Sigma^t L(s), R] \to [\Sigma^t L(s+1), R] \\
		\tr^*_\otimes: [\Sigma^t L(s), gl_1 R] \to [\Sigma^t L(s+1), gl_1 R]
	\end{gather*}
	\noindent induced by the transfer map $L(s+1) \to L(s)$.  By (\ref{prop:SameGroups}), we can identify the source and target of these two maps if $t \geq 1$.  We want to show that these two transfers agree in an expected range; our starting point is the following proposition.  
	
\begin{prop}[{\cite[Cor.~5.2.3]{MS16}}]
	Let $m \geq 1$.  There is a functorial equivalence of spectra \[\tau_{[m,2m-1]} gl_1 R \simeq \tau_{[m,2m-1]} R.\]  
\end{prop}

	Using this proposition, we can show
\begin{prop}
	Let $t \geq 1$ and $x \in [\Sigma^t L(s), R]$.  Then \[\tr^*_\oplus(x) \equiv \tr^*_\otimes(x) \pmod{[\Sigma^t L(s+1), \tau_{\geq 2t} R]}.\]  
\end{prop}
\begin{proof}
	Let $t \geq 1$, and consider the following diagram

\begin{center}
\begin{tikzcd}
	{[\Sigma^t L(s), R]} \ar[r] \ar[ddd, dash, bend right=70, "\cong"'] & {[\Sigma^t L(s), \tau_{< 2t} R]} \ar[r] & {[\Sigma^t L(s+1), \tau_{< 2t} R]} \ar[ddd, dash, bend left=70, xshift=20pt, "\cong"] \\
	{[\Sigma^t L(s), \tau_{\geq t} R]} \ar[r] \ar[u, "\ast"] \ar[d, dash, "\cong"'] & {[\Sigma^t L(s), \tau_{[t,2t-1]} R]} \ar[r] \ar[d, dash, "\cong"] \ar[u] \ar[dr, phantom, "\bullet"] & {[\Sigma^t L(s+1), \tau_{[t,2t-1]} R]} \ar[u] \ar[d, dash, "\cong"] \\
	{[\Sigma^t L(s), \tau_{\geq t} gl_1 R]} \ar[r] \ar[d, "\ast"'] & {[\Sigma^t L(s), \tau_{[t,2t-1]} gl_1 R]} \ar[r] \ar[d] & {[\Sigma^t L(s+1), \tau_{[t,2t-1]} gl_1 R]} \ar[d] \\
	{[\Sigma^t L(s), gl_1 R]} \ar[r] & {[\Sigma^t L(s), \tau_{< 2t} gl_1 R]} \ar[r] & {[\Sigma^t L(s+1), \tau_{< 2t} gl_1 R]}
\end{tikzcd}
\end{center}
	
	The crucial fact is that the proposition above implies that the square marked by $\bullet$ is also commutative, so the entire diagram commutes.  This implies that on an element coming from $[\Sigma^t L(s), \tau_{\geq t} R]$, the two maps $\tr^*_\oplus$ and $\tr^*_\otimes$ up to $[\Sigma^t L(s+1), \tau_{\geq 2t} R]$ as wanted.  
	
	It remains to show that the maps labeled $\ast$ are surjective.  This follows easily from Atiyah-Hirzebruch estimates since $\Sigma^t L(s)$ is $t$-connective and $\tau_{\geq t} R$ (respectively $\tau_{\geq t} gl_1 R$) has no homotopy groups below degree $t$.  
\end{proof}
	

\section{Appendix: Dyer-Lashof structures on \texorpdfstring{$\mathbb{F}_2[u]$}{F2[u]}}

\subsection{Problem introduction and examples}
	In \cite{Tot95}, Totaro posed a version of the following question:
\begin{quest}
	Classify all $E_\infty$-structures on a ring spectrum $R$ with $\pi_* R \cong \mathbb{F}_2[u]$ with $\deg u = 2$.  
\end{quest}

	As a preliminary goal, one might first want to classify all Dyer-Lashof algebra structures on $\mathbb{F}_2[u]$ that are compatible with the ring multiplication.  In this appendix, we prove the following theorem.

\begin{thm}
	\label{thm:DLStrs}
	There are exactly four Dyer-Lashof algebra structures on $\mathbb{F}_2[u]$ compatible with the multiplication.  
\end{thm}

	Here are some examples.  
\begin{exmp}[Segal]
	In \cite{Seg75}, Segal studied the generalized Eilenberg-Mac~Lane spectrum $H\mathbb{F}_2[u]$ with its usual $E_\infty$-structure.  The Dyer-Lashof operations are determined by \[Q^{2r} u = \begin{cases} u^2, & r = 1 \\ 0, & \text{otherwise}. \end{cases}\]  That is, all the operations that could be zero are zero.  
\end{exmp}

\begin{exmp}[BLLMM]
	Boyer, Lawson, Lima-Filho, Mann, and Michelsohn \cite{BLLMM93} constructed a theory that receives an infinite loop map given by the total Chern class.  In \cite{Tot95}, Totaro figured out the Dyer-Lashof operations on this infinite loop space, which is given by \[Q^{2r} u = u^{r+1}\] for all $r \geq 1$.  That is, all the operations that could be nonzero are nonzero.  
\end{exmp}

	Moreover, in the course of proving his result, Totaro showed that
\begin{prop}[{\cite[Lem.~5]{Tot95}}]
	The Dyer-Lashof structure on the BLLMM theory is uniquely determined by $Q^4 u \neq 0$.  
\end{prop}
\begin{rem}
	This result is stated and proved for all primes in \cite{Tot95}.
\end{rem}

	Theorem \ref{thm:DLStrs} is a generalization of this proposition.  
	
\begin{exmp}[$\THH(\mathbb{F}_2)$]
	As a third example, consider the spectrum $\THH(\mathbb{F}_2)$.  Since $H\mathbb{F}_2$ is an $E_\infty$-ring spectrum, so is $\THH(\mathbb{F}_2)$, and B\"{o}kstedt's calculations \cite{Bok85} show that $\pi_* \THH(\mathbb{F}_2) \cong \mathbb{F}_2[u]$.  

	Moreover, we can compute its Dyer-Lashof operations, following Lawson\footnote{\url{mathoverflow.net/q/207343/}}.  The idea is to write $\THH(\mathbb{F}_2) \simeq H\mathbb{F}_2 \wedge_{H\mathbb{F}_2 \wedge H\mathbb{F}_2} H\mathbb{F}_2$, and consider the $\Tor$-spectral sequence \[\Tor^{*,*}_{\pi_*(H\mathbb{F}_2 \wedge H\mathbb{F}_2)}(\mathbb{F}_2, \mathbb{F}_2) \Rightarrow \pi_* \THH(\mathbb{F}_2).\]  The $\Tor$-groups form an exterior algebra generated by the $\{\bar{\xi}_i \mid i \geq 1\}$, and this spectral sequence is compatible with the Dyer-Lashof operations in the sense that we can read off the Dyer-Lashof operations for $\THH(\mathbb{F}_2)$ from the Dyer-Lashof operations in the dual Steenrod algebra $\pi_*(H\mathbb{F}_2 \wedge H\mathbb{F}_2)$ and taking indecomposables.  
	
	In \cite[Thm.~III.2.4]{BMMS86}, Steinberger computed these operations in the dual Steenrod algebra.  In particular, we have \[Q^r \xi_1 = [t^{r+1}] \left(\sum_{i=0}^\infty \xi_i t^{2^i-1} \right)^{-1}.\]  Therefore, passing to indecomposables, we have \[Q^{2r} \overline{\xi_1} = \begin{cases} \overline{\xi_{s+1}}, & \text{if } r = 2^s - 1 \\ 0, & \text{otherwise}. \end{cases}\]  Consequently, in $\pi_* \THH(\mathbb{F}_2)$ we have \[Q^{2r} u = \begin{cases} u^{r+1}, & r = 2^s - 1 \\ 0, & \text{otherwise}. \end{cases}\]
\end{exmp}

\subsection{The proof}
	We now begin the proof of theorem \ref{thm:DLStrs}.  Let us write $Q^{2r} u = \epsilon_r u^{r+1}$ for some structure constants $\epsilon_r \in \mathbb{F}_2$.  For example, the general properties of modules over the Dyer-Lashof algebra gives $\epsilon_0 = 0$ and $\epsilon_1 = 1$.  Totaro's result classifies all Dyer-Lashof structures on $\mathbb{F}_2[u]$ with $\epsilon_2 = 1$ -- there is only one of them -- and so it remains to study the case $\epsilon_2 = 0$.  In fact, Totaro's proof shows a bit more:
	
\begin{prop}
	If $\epsilon_2 = 0$, then $\epsilon_r = 0$ for all even $r$.  
\end{prop}

	From now on we assume $\epsilon_2 = 0$.  The next step is to study $\epsilon_3$.  
\begin{prop}
	If $r \geq 3$ is odd, then $\epsilon_3 \epsilon_r = \epsilon_{2r+1}$.  
\end{prop}
\begin{proof}
	Consider the Adem relation \[Q^{8r} Q^6 u = \sum_{4r \leq 2i < 4r + 3} \binom{2i-7}{4i-8r} Q^{8r+6-2i} Q^{2i} u = \binom{4r-7}{0} Q^{4r+6} Q^{4r} u + \binom{4r-5}{4} Q^{4r+4} Q^{4r+2} u.\]  The left hand side is 
	\begin{equation*}
		Q^{8r} Q^6 u = \epsilon_3 \left(\sum_{i_1 + i_2 + i_3 + i_4 = 4r} \epsilon_{i_1} \epsilon_{i_2} \epsilon_{i_3} \epsilon_{i_4} \right) u^{4r+4} \equiv \epsilon_3 \epsilon_r^4 u^{4r+4}.
	\end{equation*}
	
	For the right hand side, we have 
	\begin{align*}
		Q^{4r+6} Q^{4r} u &= Q^{4r+6} (\epsilon_{2r} u^{2r+1}) = 0 \\
		Q^{4r+4} Q^{4r+2} u &= \epsilon_{2r+1} \left(\sum_{i_1 + \cdots + i_{2r+2} = 2r + 2} \epsilon_{i_1} \cdots \epsilon_{i_{2r+2}} \right) u^{4r+4} \equiv \epsilon_{2r+1} u^{4r+4}.
	\end{align*}

	So $\epsilon_3 \epsilon_r^4 = \epsilon_{2r+1}$.  
\end{proof}

	We consider the two choices for $\epsilon_3$ separately.  
	
\subsection{The subcase $\epsilon_3 = 0$}
	Observe that if $\epsilon_3 = 0$, then $\epsilon_r = 0$ for all $r \equiv 3 \pmod{4}$.  This is the starting point of the following generalization.  

\begin{prop}
	For all $s \geq 1$,
\begin{enumerate}[(i)]
	\item If $\epsilon_r = 0$ for all $r \not\equiv 1 \pmod{2^s}$ and $\epsilon_{2^s + 1} = 0$, then $\epsilon_r = 0$ for all $r \not\equiv 1 \pmod{2^{s+1}}$.  
	\item If $\epsilon_r = 0$ for all $r \not\equiv 1 \pmod{2^{s+1}}$, then $\epsilon_{2^{s+1} + 1} = 0$.  
\end{enumerate}
\end{prop}

\begin{cor}
	If $\epsilon_2 = \epsilon_3 = 0$, then $\epsilon_r = 0$ for all $r > 1$. 
\end{cor}

	Before proving the proposition, we define and study the quantity \[\epsilon_{I(s,d)} := \sum_{i_1 + \cdots + i_s = d} \epsilon_{i_1} \cdots \epsilon_{i_s}\] more carefully.  The idea is that permuting the $\epsilon_i$'s in the sum leads to lots of repetition which cancels out modulo $2$.  Concretely, a term of the form $\epsilon_{m_1}^{s_1} \cdots \epsilon_{m_n}^{s_n}$, $s = s_1 + \cdots + s_n$, with the $m_l$'s distinct, is repeated $\binom{s}{s_1, \ldots, s_n}$ times.  The other ingredient concerns the parity of multinomial coefficients.  It turns out $\binom{s}{s_1, \ldots, s_n}$ is even unless each $s_l$ is a sum of distinct summands of the binary expansion of $s$.  In summary,
\begin{lem}
		Let $s = \sum_i b_i 2^i$, $b_i \in \{0,1\}$, be the binary expansion of $s$.  Then \[\epsilon_{I(s,d)} = \sum_{\sum_{b_i = 1} a_i 2^i = d} \epsilon_{a_i}^{2^i}.\]
\end{lem}

\begin{proof}[Proof of proposition]
\mbox{}
\begin{enumerate}[(i)]
	\item Let $r \equiv 2^s + 1 \pmod{2^{s+1}}$, $r > 2^s + 1$, and consider the Adem relation \[Q^{4r - 2^{s+1}} Q^{2^{s+1} + 2} u = \sum_{i=r-2^{s-1}}^r \binom{2i-2^{s+1}-3}{4i-4r+2^{s+1}} Q^{4r+2-2i} Q^{2i} u.\]  
		
		The left hand side is zero since $Q^{2^{s+1}+2} u = \epsilon_{2^s+1} u^{2^s+2} = 0$.  
	
		For the right side, note that $Q^{2i} u = \epsilon_i u^{i+1}$, so the term would be zero unless $i \equiv 1 \pmod{2^s}$.  There is exactly one index $i = r$ satisfying this condition within the range of summation, so the right hand side is \[\binom{2r - 2^{s+1} - 3}{2^{s+1}} \epsilon_r \epsilon_{I(r+1,r+1)} \equiv \epsilon_r.\]  Here, we used the fact that the binomial coefficient is odd which follows from $2r - 2^{s+1} - 3 \equiv -1 \pmod{2^{s+2}}$ and Lucas' theorem. 
		
		This proves $\epsilon_r = 0$.  
		
	\item Consider the Adem relation \[Q^{6 \cdot 2^{s+1} + 4} Q^{2^{s+2} + 2} u = \sum_{i=3 \cdot 2^s + 1}^{2^{s+2}} \binom{2i - 2^{s+2} - 3}{4i - 6 \cdot 2^{s+1} - 4} Q^{2^{s+4} + 6 - 2i} Q^{2i} u.\]    
	
	For the left side, we have
	\begin{equation*}
		Q^{6 \cdot 2^{s+1} + 4} Q^{2^{s+2} + 2} u = \epsilon_{2^{s+1} + 1} \epsilon_{I(2^{s+1} + 2, 3 \cdot 2^{s+1} + 2)} \equiv \epsilon_{2^{s+1} + 1} (\epsilon_1^{2^{s+1}} \epsilon_{2^{s+1}+1}^2) \equiv \epsilon_{2^{s+1} + 1}.
	\end{equation*}
	
	For the right side, note that $Q^{2i} u = \epsilon_i u^{i+1}$, so the term would be zero unless $i \equiv 1 \pmod{2^{s+1}}$.  No such index $i$ lies within the range of summation, so the sum is zero.  
	
	We conclude $\epsilon_{2^{s+1} + 1} = 0$.  
\end{enumerate}
\end{proof}

\subsection{The subcase $\epsilon_3 = 1$}
	Partition the set of odd positive integers by the equivalence relation generated by $x \sim y$ if $y = 2x + 1$.  Each equivalence class is labelled by an even integer obtained by repeatedly subtracting one and dividing by two until we reach an even integer.  In the case $\epsilon_3 = 1$, the calculation of the Adem relation for $Q^{2r} Q^6 u$ show that if $x$ and $y$ belong in the same equivalence class, then $\epsilon_x = \epsilon_y$.  We denote by $\zeta_{\bar{x}}$ for the common value of $\epsilon_x$, where $\bar{x}$ is the even integer associated to the equivalence class containing $x$.  
	
	For example, $\zeta_0 = 1$.  
	
\begin{prop}
	Assume $\epsilon_3 = 1$, and given $r$ even suppose $\zeta_s = \zeta$ for all $2 \leq s < r$.  Then $\zeta_r = \zeta$.  
	
	Therefore, $\zeta_s = \zeta$ for all $s \geq 2$.  
\end{prop}
\begin{proof}
	Consider the Adem relation \[Q^{8r + 4} Q^{10} u = \sum_{i=2r+1}^{2r+3} \binom{2i-11}{4i-8r-4} Q^{8r+14-2i} Q^{2i} u.\]  The left side is \[Q^{8r+4} Q^{10} u = \epsilon_5 \epsilon_{I(6,4r+2)} = \epsilon_5 \sum_{i=1}^r \epsilon_i^4 \epsilon_{2(r-i)+1}^2.\]  
	
	Each factor occurring in each term of the sum is either $\zeta_0 = 1$ or $\zeta$.  The term $\epsilon_i^4 \epsilon_{2(r-i)+1}^2$ is $1$ iff $i + 1$ and $2(r-i) + 2$ are powers of two.  This means $i + 1 = 2^a$ and $r - i + 1 = 2^b$, so that $r + 2 = 2^a + 2^b$.  Provided $a \neq b$, these instances occur in pairs: $(i+1, r-i+1) = (2^a, 2^b)$ and $(i+1, r-i+1) = (2^b, 2^a)$, and by the uniqueness of the binary expansion they occur twice or not at all.  In either of these cases we have $\sum_{i=1}^r \epsilon_i^4 \epsilon_{2(r-i)+1}^2 \equiv \frac{r}{2} \zeta$.  The remaining case is when $r + 2$ is a power of two.  Then the sum evaluates to $1 + (\frac{r}{2} - 1) \zeta \equiv 1$.  To summarize, \[Q^{8r+4} Q^{10} u = \begin{cases} \frac{r}{2} \zeta, & r + 2 \text{ not a power of two}, \\ \zeta, & r + 2 \text{ is a power of two}. \end{cases}\]
	
	For the right hand side, since we require $i$ to be odd for a nonzero term, there are only two terms to consider: \[Q^{8r+4} Q^{10} u = Q^{4r+12} Q^{4r+2} u + \binom{4r-5}{8} Q^{4r+8} Q^{4r+6} u.\]  We have
	\begin{align*}
		Q^{4r+12} Q^{4r+2} u &= \epsilon_{2r+1} \epsilon_{I(2r+2, 2r+6)} \\
		&= \epsilon_{2r+1} ((2r+2) \epsilon_5 + \binom{2r+2}{2} \epsilon_3^2) \\
		&= \epsilon_{2r+1} = \zeta_r, \\
		Q^{4r+8} Q^{4r+6} u &= \epsilon_{2r+3} \epsilon_{I(2r+4,2r+4)} \\
		&= \epsilon_{2r+3}.	
	\end{align*}
		
	We have the equation \[\frac{r}{2} \zeta = \zeta_r + \binom{4r-5}{8} \epsilon_{2r+3}.\]  If $r \equiv 0 \pmod{4}$, then this becomes $0 = \zeta_r + \zeta_{r/2}$, so $\zeta_r = \zeta_{r/2} = \zeta$.  Otherwise, $r \equiv 2 \pmod{4}$, and we have $\zeta = \zeta_r$.  This finishes the proof.  	
\end{proof}

	There are two remaining cases to consider for the common value $\zeta$.  If $\zeta = 0$, then we have \[\epsilon_r = \begin{cases} 1 & r \text{ one less than a power of two} \\ 0 & \text{otherwise}. \end{cases}\]  If $\zeta = 1$, then we have \[\epsilon_r = \begin{cases} 1 & r \text{ odd}, \\ 0 & r \text{ even}. \end{cases}.\]  We need to check which of these are valid solutions to the system of equations determined by the Adem relations.  We know the former is valid because it is the Dyer-Lashof structure on $\pi_* \THH(\mathbb{F}_2)$.  To show that the latter is also valid, we need to do more combinatorial arguments involving binomial coefficients.  First,
	
	\begin{lem}
		Suppose that $A$ is an $\mathbb{F}_2$-algebra and its module of indecomposables $QA$ is equipped with the structure of a module over the Dyer-Lashof algebra.  Then there is a unique structure of an algebra over the Dyer-Lashof algebra on $A$, given by the Cartan formula.  
	\end{lem}

	We now need to check the Adem relations are satisfied.  
	
\begin{prop}
	The Dyer-Lashof structure determined by \[\epsilon_r = \begin{cases} 1, & r \text{ odd} \\ 0, & r \text{ even} \end{cases}\] is valid.
\end{prop}
\begin{proof}
	In terms of the structure constants, the Adem relation we need to check is \[\epsilon_s \sum_{i_1 + \cdots + i_{s+1} = r} \epsilon_{i_1} \cdots \epsilon_{i_{s+1}} = \sum_{i=s+1}^{r-s-1} \binom{i-s-1}{2i-r} \epsilon_i \sum_{j_1 + \cdots + j_{i+1} = r + s - i} \epsilon_{j_1} \cdots \epsilon_{j_{i+1}}.\]  

	In the case $\epsilon_r = 1$ for all odd $r$, the expression \[\sum_{i_1 + \cdots + i_{s+1} = r} \epsilon_{i_1} \cdots \epsilon_{i_{s+1}}\] counts the number of ways to write $r$ as a sum of $(s+1)$ positive odd integers.  Equivalently, it is the number of ways to write $r + s + 1$ as a sum of $(s+1)$ positive even integers, or the number of ways to write $\frac{r+s+1}{2}$ as a sum of $(s+1)$ positive integers.  In combinatorics, this is known as the number of compositions.  Moreover, we know that the number of compositions of $\frac{r+s+1}{2}$ into exactly $(s+1)$ parts is given by $\binom{\frac{r+s+1}{2} - 1}{s}$.  
	
	There are two Adem relations to prove, depending on the parity of $s$.  If $s$ is even, we need to show \[0 \equiv \sum_{i \text{ odd}} \binom{i-s-1}{2i-r} \binom{\frac{r+s-i+i+1}{2} - 1}{i}.\]  If $s$ is odd, we need to show \[\binom{\frac{r+s+1}{2} - 1}{s} \equiv \sum_{i \text{ odd}} \binom{i-s-1}{2i-r} \binom{\frac{r+s-i+i+1}{2} - 1}{i}.\]  
	
	Let us first show \[\sum_{i \text{ odd}} \binom{i-s-1}{2i-r} \binom{\frac{r+s+1}{2} - 1}{i} \equiv 0\] if $s$ is even.  Since $s$ is even, the second binomial coefficient is nonzero iff $r$ is odd.  Let us write $r = 2p + 1$, $s = 2q$, and $i = 2l + 1$.  Then we need to show \[\sum_l \binom{2l - 2q}{4l-2p + 1} \binom{p+q}{2l+1} \equiv 0.\]  This is clear, since $\binom{2l - 2q}{4l-2p+1} \equiv 0$ by looking at the one's bit.  
	
	For the second relation, where $s = 2q-1$ is odd and thus $r = 2p$ is even, we need to show \[\binom{p+q-1}{2q-1} \equiv \sum_l \binom{2l-2q+1}{4l-2p+2} \binom{p+q-1}{2l+1}.\]  
	
	It is convenient to make the substitutions $k = l - q$ and $n = p - 2q - 1$ at the outset, so that we need to prove \[\sum_k \binom{2k+1}{4k-2n} \binom{n+3q}{2k+2q+1} \equiv \binom{n+3q}{2q-1}.\]  The constraint $r > 2s$ became $p \geq 2q$, which became $n \geq 0$ and $q \geq 1$ here.  Observing that \[\binom{2k+1}{4k-2n} \equiv \binom{2k}{4k-2n} \equiv \binom{k}{2k-n} = \binom{k}{n-k},\] we end up needing to show \[\sum_k \binom{k}{n-k} \binom{n+3q}{2k+2q+1} \equiv \binom{n+3q}{2q-1}.\]  A generating function for the expression on the left is \[a_{n,q} := [w^n z^{2q-1}] \, \frac{(1+z)^{n+3q}}{w+w^2+z^2}.\]  Let \[b_{n,q} := \binom{n+3q}{2q-1}.\]  Our goal is to show that $a_{n,q} \equiv b_{n,q} \pmod{2}$ for $n, q \geq 0$.  I learnt the following proof from \url{math.stackexchange.com/q/3004244}.  
	
	Observe that
	\begin{align*}
		b_{n,q} &= \binom{n+3q}{2q-1} = \binom{n+3q-1}{2q-1} + \binom{n+3q-1}{2q-2} \\
		&= \left( \binom{n+3q-2}{2q-1} + \binom{n+3q-2}{2q-2} \right) + \left( \binom{n + 3q - 2}{2q-2} + \binom{n + 3q - 2}{2q-3}\right) \\
		&\equiv \binom{n+3q-2}{2q-1} + \binom{n+3q-2}{2q-3} \\
		&= b_{n-2,q} + b_{n+1,q-1}.
	\end{align*}
	
	Rearranging, we have \[b_{n,q} \equiv b_{n-1,q+1} + b_{n-3,q+1}.\]  
	
	We now show that the $a_{n,q}$'s satisfy a similar recurrence relation.  For $n \geq 3$, we have
	\begin{align*}
		a_{n,q} + a_{n-1,q+1} + a_{n-3,q+1} &= [w^n z^{2q-1}] \, \frac{(1+z)^{n+3q}}{w+w^2+z^2} + \frac{w}{z^2} \cdot \frac{(1+z)^{n+3q+2}}{w+w^2+z^2} + \frac{w^3}{z^2} \cdot \frac{(1+z)^{n+3q}}{w+w^2+z^2} \\
		&= [w^n z^{2q-1}] \, \frac{(1+z)^{n+3q}}{w+w^2+z^2} \left(1 + \frac{w(1+z)^2}{z^2} + \frac{w^3}{z^2} \right) \\
		&= [w^n z^{2q-1}] \, \frac{(1+z)^{n+3q}}{w+w^2+z^2} \left( \frac{w + z^2 + w^3 + wz^2}{z^2} \right) \\
		&= [w^n z^{2q-1}] \, \frac{(1+z)^{n+3q}}{w+w^2+z^2} \left( \frac{(1+w)(w+w^2 + z^2) - 2w^2}{z^2} \right) \\
		&\equiv [w^n z^{2q-1}] \, \frac{(1+z)^{n+3q}}{w+w^2 + z^2} \left( \frac{(1+w)(w+w^2+z^2)}{z^2} \right) \\
		&= [w^n z^{2q+1}] \, (1+z)^{n+3q} (1 + w) \\
		&= 0.
	\end{align*}
	
	Since $a_{n,q}$ and $b_{n,q}$ satisfy the same third-order recurrence relation in $n$, it remains to show that for all $q \geq 0$,
	\begin{align*}
		a_{0,q} &\equiv b_{0,q} \\
		a_{1,q} &\equiv b_{1,q} \\
		a_{2,q} &\equiv b_{2,q}.
	\end{align*}
	
	Repeatedly applying the recurrence for binomial coefficients, we have 
	\begin{align*}
		a_{0,q} + b_{0,q} &= \binom{3q}{2q+1} + \binom{3q}{2q-1} \equiv \binom{3q+2}{2q+1}, \\
		a_{1,q} + b_{1,q} &= \binom{3q+1}{2q+3} + \binom{3q+1}{2q-1} \equiv \binom{3q+5}{2q+3}, \\
		a_{2,q} + b_{2,q} &= \left( \binom{3q+2}{2q+5} + \binom{3q+2}{2q+3} \right) + \binom{3q+2}{2q-1}.
	\end{align*}
	
	The following special cases of Lucas' theorem is useful.  
\begin{lem}
	$\binom{2a}{2b} \equiv \binom{a}{b}$, $\binom{2a}{2b+1} \equiv 0$, and $\binom{2a+1}{2b} \equiv \binom{2a+1}{2b+1} \equiv \binom{a}{b}$.  
\end{lem}
	
	We first prove that $\binom{3q+2}{2q+1}$ is even by induction on $q$.  If $q = 0$, then this is just $\binom{2}{1} = 2$.  For the induction step, observe that if $q$ is even, then $\binom{3q+2}{2q+1}$ has the form ``even choose odd'', which is even.  Otherwise, if $q$ is odd, write $q = 2q_1 + 1$ with $q_1 < q$.  Then we have \[\binom{3q+2}{2q+1} = \binom{6q_1 + 5}{4q_1 + 2} \equiv \binom{3q_1 + 2}{2q_1 + 1}\] which is even by the induction hypothesis.  
	
	By the way, this result -- $\binom{3q+2}{2q+1}$ is even -- will be used several times in the sequel.  
	
	Second, we prove that $\binom{3q+5}{2q+3}$ is even.  If $q$ is even, then $\binom{3q+5}{2q+1}$ has the form ``even choose odd'', which is even.  Otherwise, if $q$ is odd, write $q = 2q_1 + 1$. Then we have \[\binom{3q+5}{2q+3} = \binom{6q_1 + 8}{4q_1 + 5} \equiv 0.\]  
	
	It remains to prove that \[a_{2,q} + b_{2,q} \equiv \binom{3q+2}{2q+5} + \binom{3q+2}{2q+3} + \binom{3q+2}{2q-1}\] is even.  If $q = 0$, this sum is $0$, which is even.  In fact, if $q$ is even, then each term in this sum has the form ``even choose odd'', which is even.  Otherwise, if $q$ is odd, write $q = 2q_1 + 1$ with $q_1 < q$.  Then we have \[\binom{3q+2}{2q+5} + \binom{3q+2}{2q+3} + \binom{3q+2}{2q-1} = \binom{6q_1 + 5}{4q_1 + 7} + \binom{6q_1 + 5}{4q_1 + 5} + \binom{6q_1 + 5}{4q_1 + 1} \equiv \binom{3q_1 + 2}{2q_1 + 3} + \binom{3q_1 + 2}{2q_1 + 2} + \binom{3q_1 + 2}{2q_1}.\]  
	
	If $q_1 = 2q_2$ is even, then we consider \[\binom{6q_2 + 2}{4q_2 + 3} + \binom{6q_2 + 2}{4q_2 + 2} + \binom{6q_2 + 2}{4q_2} \equiv 0 + \binom{3q_2 + 1}{2q_2 + 1} + \binom{3q_2 + 1}{2q_2}.\]  If $q_2$ is even, then each term in this expression has the form ``even choose odd'', so the sum is even.  If $q_2 = 2q_3 + 1$ is odd, then we have \[\binom{6q_3 + 4}{2q_3 + 1} + \binom{6q_3 + 4}{4q_3 + 2} \equiv 0 + \binom{3q_3 + 2}{2q_3 + 1},\] which is something we've previously shown was even.  
	
	If $q_1 = 2q_2 + 1$ is odd, then we consider \[\binom{6q_2 + 5}{4q_2 + 5} + \binom{6q_2 + 5}{4q_2 + 4} + \binom{6q_2 + 5}{4q_2 + 2} \equiv \binom{3q_2 + 2}{2q_2 + 2} + \binom{3q_2 + 2}{2q_2 + 2} + \binom{3q_2 + 2}{2q_2 + 1} \equiv \binom{3q_2 + 2}{2q_2 + 1},\] which we've previously shown was even.  
	
	This concludes the proof that $a_{n,q} = b_{n,q}$, which means that the Adem relations in the case $\epsilon_r = 1$ for all $r$ odd is verified.  Thus we obtain another valid Dyer-Lashof structure on $\mathbb{F}_2[u]$.  
\end{proof}

\begin{rem}
	We could give a similar combinatorial proof for the validity of the Dyer-Lashof structure determined by $\epsilon_r = 1$ iff $r$ is one less than a power of two, but since we know $\THH(\mathbb{F}_2)$ exists we are spared from another such tedious argument.  
\end{rem}

	In summary, the four possible Dyer-Lashof structures on $\mathbb{F}_2[u]$ are characterized by:
\begin{enumerate}[(i)]
	\item $Q^{2r} u = 0$ unless $r = 1$.
	\item $Q^{2r} u = 1$ for all $r \geq 1$.
	\item $Q^{2r} u = 0$ unless $r$ is one less than a power of two.  
	\item $Q^{2r} u = 0$ unless $r$ is odd.  
\end{enumerate}

	Moreover, (i) is realized by the Eilenberg-Mac~Lane spectrum $H\mathbb{F}_2[u]$, (ii) is realized by the BLLMM spectrum, and (iii) is realized by $\THH(\mathbb{F}_2)$.  
	
\begin{quest}
	Is there an $E_\infty$-$H$-algebra $A$ with $\pi_* A \cong \mathbb{F}_2[u]$ and Dyer-Lashof structure given by (iv)?
\end{quest}

\subsection{The obstruction spectral sequence}
	The next question to ask to whether there is a unique $E_\infty$-$H$-algebra with these prescribed Dyer-Lashof operations.  The main tool to study this is again the Goerss-Hopkins-Miller obstruction spectral sequence \cite{GH04}, introduced in section \ref{sec:GHM}.
	
\begin{thm}
	Let $A$ and $B$ be $E_\infty$-$H$-algebras, and assume that a basepoint $\phi \in \CAlg_H(A,B)$ is given.  Then,
	\begin{enumerate}[(a)]
		\item There is a Bousfield-Kan spectral sequence \[E_2^{s,t} \cong D_\mathcal{R}^s(\pi_* A, \Omega^t \pi_* B) \Rightarrow \pi_{t-s}(\CAlg_H(A,B),\phi).\]
		\item There is a Grothendieck spectral sequence \[E_2^{p,q} \cong \Ext_{\mathcal{U}(\Gamma)}(D_q(\Gamma), M) \Rightarrow \mathcal{D}_{\mathcal{R}}^{p+q}(\Gamma, M).\]  
	\end{enumerate}
\end{thm}

	In our situation, we have $\Gamma = M = \mathbb{F}_2[u]$ in (b). The first step is to compute the ordinary Andr\'{e}-Quillen homology group.  
\begin{lem}
	\[D_q(\mathbb{F}_2[u]) \cong \begin{cases}
		\mathbb{F}_2[u] \, du, & q = 0 \\
		0, & q > 0.
	\end{cases}\]
\end{lem}
\begin{proof}
	$\mathbb{F}_2[u]$ is smooth of dimension one over $\mathbb{F}_2$.  
\end{proof}

	Consequently, the spectral sequence of (b) degenerates, and the $E_2$-page in (a) is given by \[E_2^{s,t} \cong \Ext^s_{\mathcal{U}(\mathbb{F}_2[u])}(\mathbb{F}_2[u] \, du, \Omega^t \mathbb{F}_2[u]).\]  
	
	We articulate the Dyer-Lashof structure on $\mathbb{F}_2[u] \, du$.  
\begin{lem}
	$Q^{2r}(du) = \epsilon_r (r+1) u^r \, du$.
\end{lem}
\begin{proof}
	$Q^{2r}(du) = d(Q^{2r} u) = d(\epsilon_r u^{r+1}) = \epsilon_r (r+1) u^r \, du$.
\end{proof}

\begin{rem}
	Observe that except in the case of the BLLMM theory, $Q^{2r}(du) = 0$ for all $r \geq 1$ since $\epsilon_r = 0$ if $r$ is odd.  
\end{rem}

	The obstructions to existence and uniqueness of $E_\infty$-structures lie in the $(-1)$ and $(-2)$-stems.  
\begin{quest}
	Are there any obstructions to existence and uniqueness?
\end{quest}

	Once we have pinned down exactly which rings exist, we can also ask
\begin{quest}
	What are the strict units of these rings?  
\end{quest}

	We hope to return to these questions in later work.  
	
\bibliographystyle{amsalpha}
\bibliography{strict-units}

\end{document}